\documentclass[12pt,leqno]{article}

\usepackage{slashbox}
\usepackage{mathrsfs}
\usepackage{amsmath}
\usepackage{amsthm}
\usepackage{amssymb}

\usepackage{amsfonts}
\usepackage{epsf}
\usepackage{graphicx}
\usepackage{hangcaption}
\usepackage{CJK}
\usepackage{hyperref}
\usepackage{amssymb}
\numberwithin{equation}{section}
\renewcommand{\bf}{\bfseries}

\newtheorem{thm}{Theorem}

\newtheorem{lem}{Lemma}

 \numberwithin{equation}{section}

\begin{document}
\title{Asymptotics of the Discrete Chebyshev Polynomials}
\author{J. H. Pan$^\text{a}$ and R. Wong$^\text{b}$}
\date{}
\maketitle  \noindent \emph{$^\text{a}$ Department of Mathematics,
City University of Hong
Kong, Tat Chee Avenue, Kowloon, Hong Kong\\
$^\text{b}$ Liu Bie Ju Centre for Mathematical Sciences, City
University of Hong Kong, Tat Chee Avenue, Kowloon, Hong Kong\\}

\begin{abstract}
The discrete Chebyshev polynomials $t_n(x,N)$ are orthogonal with
respect to a distribution, which is a step function with jumps one
unit at the points $x=0,1,\cdots, N-1$, $N$ being a fixed positive
integer. By using a double integral representation, we have recently
obtained asymptotic expansions for $t_{n}(aN,N+1)$ in the double
scaling limit, namely, $N\rightarrow\infty$ and $n/N\rightarrow b$,
where $b\in (0,1)$ and $a\in(-\infty,\infty)$; see [Studies in Appl.
Math. \textbf{128} (2012), 337-384]. In the present paper, we
continue to investigate the behaviour of these polynomials when the
parameter $b$ approaches the endpoints of the interval $(0,1)$.
While the case $b\rightarrow 1$ is relatively simple (since it is
very much like the case when $b$ is fixed), the case $b\rightarrow
0$ is quite complicated. The discussion of the latter case is
divided into several subcases, depending on the quantities $n$, $x$
and $xN/n^2$, and different special functions have been used as
approximants, including Airy, Bessel and Kummer functions.
\end{abstract}

\newpage

\section{INTRODUCTION}
The discrete Chebyshev polynomials $t_n(x,N)$ can be defined as a
special case of the Hahn polynomials \cite[p.174]{Richard Beal}
\begin{equation}\label{definition of Q}
\begin{split}
Q_n(x;\alpha,\beta,N) &= {_3F_2}(-n,-x,n+\alpha+\beta+1;-N,\alpha+1;1) \\
   &= \sum_{k=0}^{n}\frac{(-n)_k(-x)_k(n+\alpha+\beta+1)_k}{(-N)_k(\alpha+1)_k
   k!}.
\end{split}
\end{equation}
With $\alpha=\beta=0$ in (\ref{definition of Q}), we have
\begin{equation}\label{relation between DCP and Hahn}
\begin{split}
t_n(x,N)&=(-1)^n(N-n)_n Q_n(x;0,0,N-1)\\
&=(-1)^n(N-n)_n\sum_{k=0}^n\frac{(-n)_k(-x)_k(n+1)_k}{(-N+1)_k k!
   k!};
\end{split}
\end{equation}
see \cite[p.176]{Richard Beal}.\\

In a recent paper \cite{pan}, we have studied the asymptotic
behaviour of $t_n(x,N+1)$ as $n$, $N\rightarrow\infty$ in such a way
that the ratios
\begin{equation}\label{a and b}
    a=x/N\qquad\text{and}\qquad b=n/N,
\end{equation}
satisfy the inequalities
\begin{equation}\label{}
    -\infty<a\leq \frac{1}{2}\qquad\text{and}\qquad 0<b<1.
\end{equation}
In view of the symmetry relation \cite[eq.(8)]{pan}
\begin{equation}\label{symmetricity}
    t_n(x,N+1)=(-1)^n t_n(N-x,N+1),
\end{equation}
our study actually covers the entire real-axis $-\infty<x<\infty$
or, equivalently, the entire parameter range $-\infty<a<\infty$. For
a discussion of the asymptotic behaviour of the Hahn polynomials,
see
\cite{lin}.\\

In the present paper, we shall investigate the behaviour of the
polynomials $t_n(x, N+1)$, when the parameter $b$ given in (\ref{a
and b}) either tends to $0$ or tends to $1$. The case $b\rightarrow
1$ turns out to be rather simple, since the ultimate expansions and
their derivations are similar to those in the case when $b$ is
fixed; see (\ref{result a>0}) for $0\leq a\leq\frac{1}{2}$ and
(\ref{result
a<0}) for $a<0$.\\

To derive asymptotic approximation for $t_n(x, N+1)$ when
$b\rightarrow 0$, we divide our discussion into several cases,
depending on the quantities $n$, $x$ and $xN/n^2$. A summary of our
findings is given in Table \ref{summary table} below.
\begin{table}[!htb]\label{summary table}
  \centering
  \caption{Asymptotic expansions in different cases when $\frac{n}{N}\rightarrow 0$ and $N\rightarrow \infty$}\label{table}
  \vspace{0.5cm}
  \noindent
  \begin{tabular}{|l||*{3}{c|}} \hline
  \backslashbox{$x$}{$xN/n^2$}
   & small & fixed & large \\\hline\hline
  large & Kummer & Airy & Bessel \\ \hline
  fixed & Kummer & Airy & series(\ref{relation between DCP and Hahn}) \\\hline
  small & Kummer / series(\ref{relation between DCP and Hahn}) & Airy & series(\ref{relation between DCP and Hahn}) \\
  \hline
\end{tabular}
\end{table}
Some explanation is needed for this table. By ``small", ``fixed" and
``large", we mean, respectively, $x=o(1)$, $x\in[\delta, M]$ and
$x\gg O(1)$ as $N\rightarrow \infty$, where $\delta$ is a small
positive number and $M$ is a large positive number. By ``Airy", we
mean an asymptotic expansion whose associated approximants are the
Airy functions $\text{Ai}(z)$ and $\text{Bi}(z)$. In the same
manner, by ``Bessel" and ``Kummer", we mean asymptotic expansions
whose associated approximants are, respectively, the Bessel function
$J_v(z)$ and the Kummer function $M(a,c,z)$. The Kummer function
used in this paper is defined by
\begin{equation}\label{Kummer function}
    \mathbf{M}(a,c,z)=\frac{1}{\Gamma(c)}M(a,c,z)=\frac{1}{\Gamma(c)}\sum_{s=0}^{\infty}\frac{(a)_s}{(c)_s}\frac{z^s}{s!},
\end{equation}
which was introduced in \cite[p.255]{Olver's book}. The case
involving Bessel functions covers an earlier result of Sharapodinov
\cite{Sharapodinov}, where, instead of Bessel functions, Jacobi
polynomials are used as an approximant. In the case when $xN/n^2$ is
large and $x$ is either fixed or small, the series in (\ref{relation
between DCP and Hahn}) is itself an asymptotic expansion (in the
generalized sense \cite[p.10]{R.Wong's book}) as
$N\rightarrow\infty$. Hence, there is
no need to seek for another asymptotic representation.\\

The arrangement of the present paper is as follows. In Section 2, we
recall the major results in \cite{pan}; to facilitate our
presentation for later sections, we also include here brief
derivations of the expansions given in \cite{pan}. In Section 3, we
consider the case when $xN/n^2$ is small and show that the
expansions are of Kummer-type. In Section 4, we consider the case
when the quantity $xN/n^2$ is fixed; that is, when $xN/n^2$ is
bounded away from zero and infinity. There are three subcases,
depending on $x$ being large, fixed or small. In all three subcases,
we will show that the expansions are of Airy-type; see Table
\ref{summary table}. The case when both quantities $x$ and $xN/n^2$
are large is dealt with in Section 5, where we show that the
expansion of $t_n(x,N+1)$ can be expressed in terms of Bessel
functions. The final section is devoted to two remaining cases,
namely, (i) when $xN/n^2$ is large and $x$ is either fixed or small, and (ii) $x<0$.\\

\section{RESULTS IN REFERENCE \cite{pan}}
In \cite{pan}, we have divided our discussion into two cases: (i)
$0\leq a\leq \frac{1}{2}$, and (ii) $a<0$. As mentioned in Section
1, by virtue of the symmetry relation (\ref{symmetricity}), these
two cases cover the entire range $-\infty<a<\infty$.\\

In case (i), we started with the integral representation
\begin{equation}\label{IR of DCP for x>0}
t_n(x,N+1)=\frac{(-1)^n}{2\pi
i}\frac{\Gamma(n+N+2)}{\Gamma(n+1)\Gamma(N-n+1)}\int_0^1\int_{\gamma_1}\frac{1}{w-1}
e^{N f(t,w)}\mathrm{d}w\mathrm{d}t,
\end{equation}
where
\begin{equation}\label{phase function f for x>0}
f(t,w)=b \ln(1-t)+(1-b)\ln t+a\ln w-a \ln(w-1)+b \ln[1-(1-t)w],
\end{equation}
and the curve $\gamma_1$ starts at $w=0$, runs along the lower edge
of the positive real line towards $w=1$, encircles the point $w=1$
in the counterclockwise direction and returns to the origin along
the upper edge of the positive real line. As a function of two
variables, the partial derivatives $\partial f/\partial t$ and
$\partial f/\partial w$ vanish at $(t,w)=(t_\pm, w_\pm)$, where
\begin{equation}\label{sets of saddle points 1}
  t_{+}=\frac{2-2a-b^2+ b\sqrt{b^2-4a+4a^2}}{2(1-a)(1+b)},\qquad
  w_+=\frac{b+\sqrt{b^2-4a+4a^2}}{2b},
\end{equation}
and
\begin{equation}\label{sets of saddle points 2}
 t_{-}=\frac{2-2a-b^2-
 b\sqrt{b^2-4a+4a^2}}{2(1-a)(1+b)},\qquad  w_-=\frac{b-\sqrt{b^2-4a+4a^2}}{2b} .
\end{equation}
For each fixed $w\in \gamma_1$, we now find a steepest descent path
of the phase function $f(t,w)$ in the variable $t$, which passes
through a saddle point $t_0(w)$ depending on $w$. (Note that not
only the saddle point $t_0(w)$, but all points $t$ on this steepest
descent path depend on $w$.) The function $f(t_0(w),w)$ is a
function of $w$ alone. To find the relevant saddle point $t_0(w)$,
we solve the equation $\partial f(t,w)/\partial t=0$, and obtain
\begin{equation}\label{saddle points for t}
    t_0(w)=\frac{2w-1+\sqrt{1+4b^2(w-1)w}}{2(1+b)w}.
\end{equation}
It can be shown that
\begin{equation}\label{temptemp123}
    t_0(w_+)=t_+,\qquad t_0(w_-)=t_-.
\end{equation}

%
Define the standard transformation $t\rightarrow\tau$ by
\begin{equation}\label{mapping t to tau}
    \tau^2=f(t_0(w),w)-f(t,w);
\end{equation}
see \cite[p.88]{R.Wong's book}. Note that we have $\tau=-\infty$
when $t=0$, and $\tau=+\infty$ when $t=1$. Furthermore, this mapping
takes $t=t_0(w)$ to $\tau=0$. Coupling $(\ref{IR of DCP for x>0})$
and $(\ref{mapping t to tau})$ gives
\begin{equation}\label{IP after the t mapping}
\begin{split}
t_n(x,N+1)=& \frac{(-1)^n}{2\pi
i}\frac{\Gamma(n+N+2)}{\Gamma(n+1)\Gamma(N-n+1)}\\
&\times \int_{-\infty}^{+\infty}\int_{\gamma_1}\frac{1}{w-1}e^{N
f({t_0(w)},w)}e^{-N\tau^2}\frac{\mathrm{d}t}{\mathrm{d}
\tau}\mathrm{d}w\mathrm{d}\tau,
\end{split}
\end{equation}
where $\gamma_1$ is the integration path of the variable $w$ in
(\ref{IR of DCP for x>0}). Note that the first variable of $f$ in
(\ref{IP after the t mapping}) is now $t_0(w)$, instead of $t$.
Hence, the phase function is a function of $w$ alone. Setting
$\partial f(t_0(w),w)/\partial w=0$, we obtain the saddle points
%
\begin{equation}\label{saddle points of w}
    w_{\pm}=\frac{b\pm\sqrt{b^2-4a+4a^2}}{2b};
\end{equation}
cf.(\ref{sets of saddle points 1}) and (\ref{sets of saddle points
2}). Motivated by (\ref{phase function f for x>0}), define the new
phase function
\begin{equation}\label{psi function}
    \psi(u)=a\ln u-a\ln(u-1)+\eta u.
\end{equation}
The saddle points of $\psi$ are given by
\begin{equation}\label{saddle points of u}
    u_\pm=\frac{\eta\pm\sqrt{\eta^2+4a\eta}}{2\eta}.
\end{equation}
To reduce the double integral in (\ref{IR of DCP for x>0}) into a
canonical form, we define the second mapping $w\rightarrow u$ by
\begin{equation}
\begin{split}
    f(t_0(w),w)&=\psi(u)+\gamma\\
    &=a \ln u-a\ln(u-1)+\eta u+\gamma\label{mapping w to u}
\end{split}
\end{equation}
with
\begin{equation}\label{the relation between w+- and u+-}
    u(w_+)=u_-,\quad\quad\quad\quad u(w_-)=u_+,
\end{equation}
where $\eta$ and $\gamma$ are real numbers depending on the
parameters $a$ and $b$ in (\ref{phase function f for x>0}). From
(\ref{mapping w to u}), we have
\begin{equation}
\begin{split}
\frac{\mathrm{d} w}{\mathrm{d} u} &= \frac{\mathrm{d}
\psi}{\mathrm{d} u}\left/\frac{\mathrm{d}
  f(t_0(w),w)}{\mathrm{d} w}\right.
 \\
   &= \frac{\eta(u-u_+)(u-u_-)w(1-w)\left[(2a-1)-\sqrt{1+4b^2w^2-4b^2w}\right]}{-2b^2u(u-1)(w-w_+)(w-w_-)}\label{dw over du, normal}
\end{split}
\end{equation}
for $u\neq u_\pm$, and by L'H$\hat{o}$pital's rule,
\begin{equation}\label{dw over du, saddle point}
    \frac{\mathrm{d} w}{\mathrm{d}
    u}=\left[\frac{\sqrt{\eta^2+4a\eta}(1-2a)(1-a)(-\eta)}{b^3\sqrt{b^2-4a+4a^2}}\right]^{1/2}
\end{equation}
for $u=u_\pm$. With the change of variable $w\rightarrow u$ defined
in (\ref{mapping w to u}), the representation in (\ref{IP after the
t mapping}) becomes
\begin{equation}\label{IP for x>0 deformed}
\begin{split}
 t_n(x,N+1)=&\frac{(-1)^n}{2\pi
i}\frac{\Gamma(n+N+2)}{\Gamma(n+1)\Gamma(N-n+1)}e^{N\gamma}\\
&\qquad\times
\int_{-\infty}^{+\infty}\int_{\gamma_u}\frac{h(u,\tau)}{u-1}e^{N
\psi(u)-N\tau^2}\mathrm{d}u\mathrm{d}\tau,
\end{split}
\end{equation}
where
\begin{equation}\label{h function}
h(u,\tau)=\frac{u-1}{w-1}\frac{\mathrm{d} w}{\mathrm{d}
u}\frac{\mathrm{d} t}{\mathrm{d} \tau},
\end{equation}
and $\gamma_u$ is the steepest descent path of $\psi(u)$ in the
$u$-plane. An asymptotic expansion, holding uniformly for $a\in
[0,\frac{1}{2}]$, was then derived in \cite{pan} by an
integration-by-parts technique. To state the result, we define
recursively $h_0(u,\tau)=h(u,\tau)$,
\begin{equation}\label{hl recursively 1}
    h_l(u,\tau) = a_l(\tau)+b_l(\tau)u-(u-u_-)(u-u_+)g_l(u,\tau)
\end{equation}
and
\begin{equation}\label{hl recursively 2}
    h_{l+1}(u,\tau)=(u-1)\left[g_l(u,\tau)+u\frac{\partial
g_l(u,\tau)}{\partial u}\right],\qquad l=0,1,2,....
\end{equation}
Furthermore, we write
\begin{equation}\label{expansion of coeffiencet an and bn}
       a_l(\tau)=\sum_{j=0}^{\infty} a_{l,j}\tau^j, \qquad       b_l(\tau)=\sum_{j=0}^{\infty}
       b_{l,j}\tau^j.
\end{equation}
The resulting expansion takes the form
\begin{equation}\label{result a>0}
\begin{split}
  t_n(x,N+1) \sim \frac{(-1)^n\Gamma(n+N+2)e^{N\gamma}}{\Gamma(n+1)\Gamma(N-n+1)\sqrt{N}}&\left\{\mathbf{M}(a N+1,1,\eta N)\sum_{l=0}^{\infty}\frac{c_l}{(\eta N)^{l}}\right.\\
   & \left.\text{ }+\mathbf{M}'(a N+1,1,\eta N)\sum_{l=0}^{\infty}\frac{d_l}{(\eta N)^{l}}\right\},
\end{split}
\end{equation}
where $\mathbf{M}(a,c,z)$ is the Kummer function defined in
(\ref{Kummer function}), and the coefficients $c_l$ and $d_l$ are
given explicitly by
\begin{equation}\label{c_l and d_l when
  a>0}
  c_l = \sum_{m=0}^la_{l-m,2m}\Gamma(m+\frac{1}{2})\eta^m,\qquad
  d_l = \sum_{m=0}^lb_{l-m,2m}\Gamma(m+\frac{1}{2})\eta^m.
\end{equation}


Case (ii) was dealt with in a similar manner. We started with the
double-integral representation
\begin{equation}\label{IR of DCP for x<0}
t_n(x,N+1)=\frac{(-1)^n}{2\pi
i}\frac{\Gamma(n+N+2)}{\Gamma(n+1)\Gamma(N-n+1)}\int_0^1\int_{\gamma_2}\frac{1}{w-1}e^{N
\widetilde{f}(t,w)}\mathrm{d}w\mathrm{d}t,
\end{equation}
where
\begin{equation}\label{phase function f for x<0}
\widetilde{f}(t,w)=b \ln(1-t)+(1-b)\ln t+a\ln(-w)-a \ln(1-w)+b
\ln[1-(1-t)w],
\end{equation}
and the integration path $\gamma_2$ starts at $w=1$, traverses along
the upper edge of the positive real line towards $w=0$, encircles
the origin in the counterclockwise direction and returns to $w=1$
along the lower edge of the positive real line.\\

Now we follow the same argument as given in Case (i), and use the
same mapping $t\rightarrow \tau$ defined in (\ref{mapping t to
tau}), except with $f(t,w)$ replaced by $\widetilde{f}(t,w)$. The
result is
\begin{equation}\label{IP after the t mapping a<0}
\begin{split}
t_n(x,N+1) =& \frac{(-1)^n}{2\pi
i}\frac{\Gamma(n+N+2)}{\Gamma(n+1)\Gamma(N-n+1)}\\
& \times \int_{-\infty}^{+\infty}\int_{\gamma_2}\frac{1}{w-1}e^{N
\widetilde{f}({t_0(w)},w)}e^{-N\tau^2}\frac{\mathrm{d} t}{\mathrm{d}
\tau}\mathrm{d}w\mathrm{d}\tau. \end{split}
\end{equation}
Because of the shape of the contour $\gamma_2$, it turns out that
the phase function $\widetilde{f}(t_0(w),w)$ has only one relevant
saddle point, namely $w_-$; see (\ref{saddle points of w}). As a
consequence, we define the mapping $w\rightarrow u$ by
\begin{equation}\label{mapping w to u when x<0}
    \widetilde{f}(t_0(w),w)=a\ln (-u)-u+\gamma
\end{equation}
with
\begin{equation}\label{}
    u(w_-)=a,
\end{equation}
where $\gamma$ is a constant depending on the parameters $a$ and $b$
in (\ref{a and b}). The final expansion is in terms of the gamma
function, and we have
\begin{equation}\label{result a<0}
      t_n(x,N+1)\sim \frac{(-1)^n\Gamma(n+N+2)N^{-a N}e^{N\gamma}}{\Gamma(n+1)\Gamma(N-n+1)\Gamma(-a
  N+1)}\sum_{l=0}^{\infty}\frac{c_l}{N^{l+\frac{1}{2}}},
\end{equation}
where $c_l$ are constants that can be given recursively.

\section{KUMMER-TYPE EXPANSION}
As $b\rightarrow 0$, some of the steps in Section 2 are no longer
valid. Let us first examine the mapping $t\rightarrow \tau$ given in
(\ref{mapping t to tau}). In this mapping, we have used the fact
that for each fixed $w\in\gamma_1$, the saddle point $t=t_0(w)$ in
(\ref{saddle points for t}) is bounded away from $t=0$ and $t=1$,
and the steepest descent path from $t=0$ to $t=1$, passing through
$t_0(w)$, is mapped onto $(-\infty,\infty)$ in the $\tau$-plane.
However, from (\ref{saddle points for t}), we have for any fixed $w$
\begin{equation}\label{property of t0(w)}
    t_0(w)=1-b+O\left(b^2\right)\qquad \text{as }b\rightarrow 0;
\end{equation}
that is, as $b\rightarrow 0$, $t_0(w)$ approaches the branch point
$t=1$ in the t-plane. Thus, the mapping (\ref{mapping t to tau}) is
no longer suitable in this case. Since we are more interested in the
neighbourhood of $t=1$, where the term $b\ln(1-t)$ in the phase
function (\ref{phase function f for x>0}) becomes singular, we
introduce the mapping
\begin{equation}\label{mapping t to tau case 1}
    f(t,w)=b \ln \tau-\tau+A,
\end{equation}
where the constant $A$ does not depend on $t$ or $\tau$ (but may
depend on $w$); see (\ref{constant A}) below. To make the mapping
$t\rightarrow \tau$ defined in (\ref{mapping t to tau case 1})
one-to-one and analytic, we prescribe $t=t_0(w)$ to correspond to
$\tau=b$, which is the saddle point of $b\ln\tau-\tau+A$; i.e.,
\begin{equation}\label{}
    \tau(t_0(w))=b.
\end{equation}
This gives
\begin{equation}\label{constant A}
    A=f(t_0(w),w)-b\ln b+b.
\end{equation}
Note that we have $\tau=0$ when $t=1$, and $\tau=+\infty$ when
$t=0$. Furthermore, from (\ref{mapping t to tau case 1}), we have
\begin{equation}\label{dt over dtau b-0 case 1}
        \begin{split}
            \frac{\mathrm{d}t}{\mathrm{d}\tau}& =\frac{b/\tau-1}{\partial f(t,w)/\partial
            t}=\frac{(\tau-b) t(1-t)[1-(1-t)w]}{\tau(1+b)w(t-t^+_0(w))(t-t^-_0(w))},\quad\quad \tau\neq
            b,\\
            \frac{\mathrm{d}t}{\mathrm{d}\tau}&=\left\{\frac{t^+_0(w)(1-t^+_0(w))[1-(1-t^+_0(w))w]}{b\sqrt{1+4b^2(w-1)w}}\right\}^{1/2}, \hspace{1.6cm}
            \tau=b,
        \end{split}
\end{equation}
where we have used L'H$\hat{o}$pital's rule for $\tau=b$ and
\begin{equation*}
    t_0^{\pm}(w)=\frac{2w-1\pm\sqrt{1+4b^2(w-1)w}}{2(1+b)w};
\end{equation*}
cf.(\ref{saddle points for t}). Coupling (\ref{IR of DCP for x>0})
and (\ref{mapping t to tau case 1}), we have, instead of (\ref{IP
after the t mapping}),
\begin{equation}\label{IP after the t mapping b-0}
\begin{split}
t_n(x,N+1)=& \frac{(-1)^{n+1}}{2\pi
i}\frac{\Gamma(n+N+2)e^nb^{-n}}{\Gamma(n+1)\Gamma(N-n+1)}\\
&\times \int_{0}^{+\infty}\int_{\gamma_1}\frac{1}{w-1}e^{N
f({t_0(w)},w)}e^{N(b\ln \tau-\tau)}\frac{\mathrm{d}t}{\mathrm{d}
\tau}\mathrm{d}w\mathrm{d}\tau.
\end{split}
\end{equation}

%
Next, let us examine the mapping (\ref{mapping w to u}). We note
that this mapping still works in the present case; the only
difference is that for fixed $n/N\in(0,1)$, the constant $\eta$ in
the mapping (\ref{mapping w to u}) is bounded away from 0, whereas
when $n/N\rightarrow 0$, $\eta$ approaches 0. More precisely, we
have $\eta\sim -b^2$. This can roughly be seen from the saddle
points (\ref{saddle points of w}) and (\ref{saddle points of u}),
together with their correspondence relation (\ref{the relation
between w+- and u+-}) under the mapping $w\rightarrow u$ given in
(\ref{mapping w to u}). To prove this rigorously, we give the
following lemma.
\begin{lem}\label{lemma1}
Let $\eta$ be the constant defined in the mapping (\ref{mapping w to
u}). If $a/b^2=o(1)$ (or, equivalently, $xN/n^2=o(1)$), then we have
\begin{equation}\label{eta and b in lemma}
    \eta=-b^2+o(b^2) \quad\text{ as }b\rightarrow 0.
\end{equation}

\end{lem}
\begin{proof}
From (\ref{saddle points for t}), we have for any fixed $w$
\begin{equation}\label{temptemp111}
    t_0(w)=\frac{1}{1+b}-\frac{b^2(1-w)}{1+b}-\frac{b^4w(w-1)^2}{1+b}+O\left(b^6
    w^2(w-1)^3\right) \quad \text{ as } b\rightarrow 0,
\end{equation}
where the term $O\left(b^4 w(w-1)^2\right)$ holds uniformly when
either $w$ or $1-w$ is small. If $a/b^2\rightarrow 0$, from
(\ref{saddle points of w}) we have $w_-=1-w_+\rightarrow 0^+$ and
both $w_\pm$ bounded. Thus,
\begin{equation}\label{temp 8}
    \begin{split}
    t_0(w_-)&=1-b+b^2(1-b)w_-+O\left(b^4 w_-^2\right),\\
    t_0(w_+)&=\frac{1}{1+b}-\frac{b^2}{1+b}w_-+O\left(b^4
    w_-^2\right).
    \end{split}
\end{equation}
To obtain $\eta$, we use (\ref{mapping w to u}) and (\ref{the
relation between w+- and u+-}). First, substituting (\ref{temp 8})
in (\ref{phase function f for x>0}) gives
\begin{equation}\label{temp lhs}
    f(t_0(w_+),w_+)-f(t_0(w_-),w_-)=2a\ln\frac{1-w_-}{w_-}-b^2+O\left(b^3\right).
\end{equation}
Here, we have made use of the fact that if $w_\pm$ lie on the upper
edge of the cut along the interval $0<w<1$, then we have
$\ln(w_\pm-1)=\ln(1-w_\pm)+\pi i$. Similarly, if $w_\pm$ lie on the
lower edge of the cut, then $\ln(w_\pm-1)=\ln(1-w_\pm)-\pi i$. Next,
we have from (\ref{psi function}) and (\ref{saddle points of u})
\begin{equation}\label{temp rhs}
    \psi (u_-)-\psi(u_+)=2a\ln\frac{1-u_+}{u_+}+\eta \sqrt{1+\frac{4a}{\eta}}.
\end{equation}
Note that $u_++u_-=1$ and $\eta<0$. Formula (\ref{eta and b in
lemma}) now follows from a combination of (\ref{mapping w to u}),
(\ref{the relation between w+- and u+-}), (\ref{temp lhs}) and
(\ref{temp rhs}).
\end{proof}

What Lemma 1 says is that if $a/b^2\rightarrow 0$ and $x=O(1)$ or
$x\gg O(1)$, then we have $|\eta|N\sim b^2 N=x b^2/a\rightarrow
\infty$, i.e., $\eta N$ is large.\\

Coupling (\ref{IP after the t mapping b-0}) and (\ref{mapping w to
u}), we have
\begin{equation}\label{IP after the t and w mapping b-0}
\begin{split}
t_n(x,N+1)=& \frac{(-1)^{n+1}\Gamma(n+N+2)e^nb^{-n}e^{N\gamma}}{2\pi
i\Gamma(n+1)\Gamma(N-n+1)}\\
&\times \int_{0}^{+\infty}\int_{\gamma_1}\frac{h(u,\tau)}{u-1}e^{N
\psi(u)}e^{N(b\ln \tau-\tau)}\mathrm{d}w\mathrm{d}\tau,
\end{split}
\end{equation}
where
\begin{equation}\label{h function in kummer case}
 h(u,\tau)=\frac{u-1}{w-1}\frac{\mathrm{d}
w}{\mathrm{d} u}\frac{\mathrm{d} t}{\mathrm{d} \tau},
\end{equation}
$\mathrm{d} w/\mathrm{d} u$ is given by (\ref{dw over du, normal})
and (\ref{dw over du, saddle point}) and $\mathrm{d} t/\mathrm{d}
\tau$ is given by (\ref{dt over dtau b-0 case 1}). Following the
same integration-by-parts procedure outlined in Section 2, we let
$h_0(u,\tau)=h(u,\tau)$ in (\ref{h function in kummer case}), and
define $h_l(u,\tau)$ recursively by (\ref{hl recursively 1}),
(\ref{hl recursively 2}) and (\ref{expansion of coeffiencet an and
bn}). The final result is

\begin{align}\label{part II result case 1}
  t_n(x,N+1) = &\frac{(-1)^{n+1}\Gamma(n+N+2)e^nb^{-n}e^{N\gamma}}{\Gamma(N-n+1)N^{n+1}}\nonumber\\
   &\times\Biggr[ \mathbf{M}(a N+1,1,\eta N)\sum_{l=0}^{p-1}\frac{c_l}{(\eta N)^{l}}+\mathbf{M}'(a N+1,1,\eta N)\sum_{l=0}^{p-1}\frac{d_l}{(\eta
   N)^{l}}+\varepsilon_p\Biggr],
\end{align}
where
\begin{eqnarray}
  c_l &=&\frac{N^{n+1}}{\Gamma(n+1)}\int_0^\infty a_l(\tau)\tau^ne^{-N\tau}\mathrm{d}\tau \sim\sum_{m=0}^\infty a_{l,m}\frac{\Gamma(n+m+1)}{\Gamma(n+1)N^m}, \label{c_l}\label{cl when a>0}\\
  d_l& =&\frac{N^{n+1}}{\Gamma(n+1)}\int_0^\infty b_l(\tau)\tau^ne^{-N\tau}\mathrm{d}\tau \sim \sum_{m=0}^\infty b_{l,m}\frac{\Gamma(n+m+1)}{\Gamma(n+1)N^m}\label{d_l}\label{d_l when
  a>0}
\end{eqnarray}
and
\begin{equation}\label{}
    \varepsilon_p=\frac{N^{n+1}}{2\pi i \Gamma(n+1)(\eta
    N)^p}\int_0^{+\infty}\int_{\gamma_1}\frac{h_p(u,\tau)}{u-1}e^{N
\psi(u)}e^{N(b\ln \tau-\tau)}\mathrm{d}w\mathrm{d}\tau.
\end{equation}
Note that since $b=n/N\rightarrow 0$, the series in (\ref{cl when
a>0}) and (\ref{d_l when a>0}) are asymptotic. Moreover, since $\eta
N$ is large in this case, (\ref{part
II result case 1}) is indeed a compound asymptotic expansion as $N\rightarrow\infty$.\\

Now, we consider the subcase in which $a/b^2\rightarrow 0$ and
$x=o(1)$. In this case, the expansion in (\ref{part II result case
1}) is still asymptotic as long as $|\eta| N\sim b^2 N\rightarrow
\infty$ or, equivalently, $n\gg O(\sqrt{N})$. If $n=O(\sqrt{N})$,
then the series in (\ref{relation between DCP and Hahn}) is itself
an asymptotic expansion as $N\rightarrow\infty$. This completes our
discussion of all three cases listed under the condition ``$xN/n^2$
small" in Table 1.

\section{AIRY-TYPE EXPANSION}
For the case $a/b^2\in[\gamma, M]$, where $\gamma$ is a small
positive number and $M$ is a large positive number, the saddle
points $w_\pm$ in (\ref{saddle points of w}) are bounded away from
the singularities $w=0$, $w=1$ and $w=\infty$, and coalesce with
each other when $a/b^2\rightarrow \frac{1}{4(1-a)}$. Therefore, to
derive an asymptotic expansion uniformly for $a/b^2\in[\gamma,M]$,
we need the cubic transformation (see (\ref{airy transformation in
case 2 b-0}) below) introduced by Chester, Friedman and Ursell
\cite{Ursell}. The resulting expansion is in terms of the Airy function $\text{Ai}(\cdot)$.\\

Following the same argument as in Section 3, we again use the
mapping from $t\rightarrow \tau$ in (\ref{mapping t to tau case 1}),
and start with the integral representation (\ref{IP after the t
mapping b-0})
\begin{equation}\label{temp 3}
\begin{split}
t_n(x,N+1)=& \frac{(-1)^{n+1}}{2\pi
i}\frac{\Gamma(n+N+2)e^nb^{-n}}{\Gamma(n+1)\Gamma(N-n+1)}\\
&\times \int_{0}^{+\infty}\int_{\gamma_1}\frac{1}{w-1}e^{N
f({t_0(w)},w)}e^{N(b\ln \tau-\tau)}\frac{\mathrm{d}t}{\mathrm{d}
\tau}\mathrm{d}w\mathrm{d}\tau,
\end{split}
\end{equation}
where the integration path is described in the line below
(\ref{phase function f for x>0}). To proceed further, we divide
$\gamma_1$ into two parts, and denote the part in the upper half of
the plane by $\gamma^{(+)}$, and the other part in the lower half of
the plane by $\gamma^{(-)}$. Recall from (\ref{phase function f for
x>0}) that the phase function in (\ref{temp 3}) is given by
\begin{equation}\label{phase function in case Airy}
\begin{split}
f(t_0(w),w)=& b \ln(1-t_0(w))+(1-b)\ln t_0(w)+a\ln w\\
& -a \ln(w-1)+b \ln[1-(1-t_0(w))w],
\end{split}
\end{equation}
which has a cut $(-\infty,1]$ in $w$-plane. Put
\begin{equation}\label{phase function in case Airy modified}
\begin{split}
    \overline{f}(t_0(w),w)=&b \ln(1-t_0(w))+(1-b)\ln t_0(w)+a\ln w\\
    &-a\ln (1-w)+b\ln[1-(1-t_0(w))w],
\end{split}
\end{equation}
which has cuts $(-\infty,0]$ and $[1,\infty)$. From (\ref{phase
function in case Airy}) and (\ref{phase function in case Airy
modified}), we have
\begin{equation}\label{temp 1}
f(t_0(w),w)=\overline{f}(t_0(w),w)+a\pi i
\end{equation}
for $w\in \gamma^{(+)}$, and
\begin{equation}\label{temp 2}
f(t_0(w),w)=\overline{f}(t_0(w),w)-a\pi i
\end{equation}
for $w\in \gamma^{(-)}$. We now make the standard transformation
\begin{equation}\label{airy transformation in case 2 b-0}
    \overline{f}(t_0(w),w)=\frac{1}{3}u^3-\zeta u+A,
\end{equation}
with the correspondence between the critical points of the two sides
prescribed by
\begin{equation}\label{airy transformation mapping of saddle points b-0}
    w_+\leftrightarrow \sqrt{\zeta},\qquad \quad w_-\leftrightarrow-\sqrt{\zeta}.
\end{equation}
If $\zeta>0$, then $\pm\sqrt{\zeta}$ are both real; if $\zeta<0$,
then $\pm\sqrt{\zeta}$ are complex conjugates and purely imaginary.
The values of $A$ and $\zeta$ can be obtained by using (\ref{airy
transformation in case 2 b-0}) and (\ref{airy transformation mapping
of saddle points b-0}). We also have
\begin{equation}\label{dw over du case 2 airy}
        \begin{split}
            \frac{\mathrm{d}w}{\mathrm{d}u}& =\frac{(u-\sqrt{\zeta})(u+\sqrt{\zeta})w(1-w)\left[(2a-1)-\sqrt{1+4b^2w^2-4b^2 w}\right]}{-2b^2(w-w_+)(w-w_-)},\quad u\neq
            \pm\sqrt{\zeta},\\
            \frac{\mathrm{d}w}{\mathrm{d} u}&=\left\{\frac{2\sqrt{\zeta}a(1-2a)(1-a)}{b^3\sqrt{b^2-4a+4a^2}}\right\}^{1/2},\hspace{1.6cm}
            u=\pm\sqrt{\zeta},
        \end{split}
\end{equation}
where we have again used L'H$\hat{o}$pital's rule for
$u=\pm\sqrt{\zeta}$. Let us first consider the case $\zeta<0$, and
deform the image of the contour $\gamma^{(+)}$ under the mapping
$w\rightarrow u$ defined in (\ref{airy transformation in case 2
b-0}) to the steepest descent path of $\frac{1}{3}u^3-\zeta u+A$ in
the $u$-plane which passes through $\sqrt{\zeta}$. We denote the
path by $C_2$. Similarly, we deform the image of the contour
$\gamma^{(-)}$, and denote the steepest descent path passing through
$-\sqrt{\zeta}$ by $C_3$; see Figure 1. Next, we consider the case
$\zeta>0$. Note that here the saddle points $\pm \sqrt{\zeta}$ are
real, and that the contours $C_2$ and $C_3$ pass through both of
them; see Figure 2. Clearly, in both cases, $C_3$ is the reflection
of $C_2$ with respect to the real axis in the $u$-plane.\\

Let $C_1$ denote the dotted curve shown in Figures 1 and 2, and
recall the identities $e^{iaN\pi}=\cos aN\pi+i\sin aN\pi$ and
$e^{-iaN\pi}=\cos aN\pi-i\sin aN\pi$. A combination of (\ref{temp
3}), (\ref{temp 1}), (\ref{temp 2}) and (\ref{airy transformation in
case 2 b-0}) then gives
\begin{equation}\label{after the airy transform}
\begin{split}
t_n(x,N+1)=& \frac{(-1)^{n}\Gamma(n+N+2)e^nb^{-n}e^{N A}}{\Gamma(n+1)\Gamma(N-n+1)}\\
&\times\left\{\frac{\cos{aN\pi}}{2\pi i}
\int_{0}^{+\infty}\int_{C_1}h_0(u,\tau)e^{N
\left(\frac{1}{3}u^3-\zeta
u \right)}e^{N(b\ln \tau-\tau)}\mathrm{d}u\mathrm{d}\tau\right.\\
&\quad\left.+\frac{\sin{aN\pi}}{2\pi }
\int_{0}^{+\infty}\int_{C_3-C_2}h_0(u,\tau)e^{N
\left(\frac{1}{3}u^3-\zeta u\right)}e^{N(b\ln
\tau-\tau)}\mathrm{d}u\mathrm{d}\tau\right\},
\end{split}
\end{equation}
where
\begin{equation}\label{h0 in airy case}
    h_0(u,\tau)=\frac{1}{w-1}\frac{\mathrm{d}t}{\mathrm{d}
\tau}\frac{\mathrm{d}w}{\mathrm{d} u},
\end{equation}
$\mathrm{d}t/\mathrm{d} \tau$, $\mathrm{d}w/\mathrm{d} u$ are given
respectively by (\ref{dt over dtau b-0 case 1}) and (\ref{dw over du
case 2 airy}).\\
\vspace{2.5cm}
\begin{figure}[h]
\centering
\begin{minipage}[t]{0.45\textwidth}
\centering
  \includegraphics[width=100pt,bb=110 0 330 170]{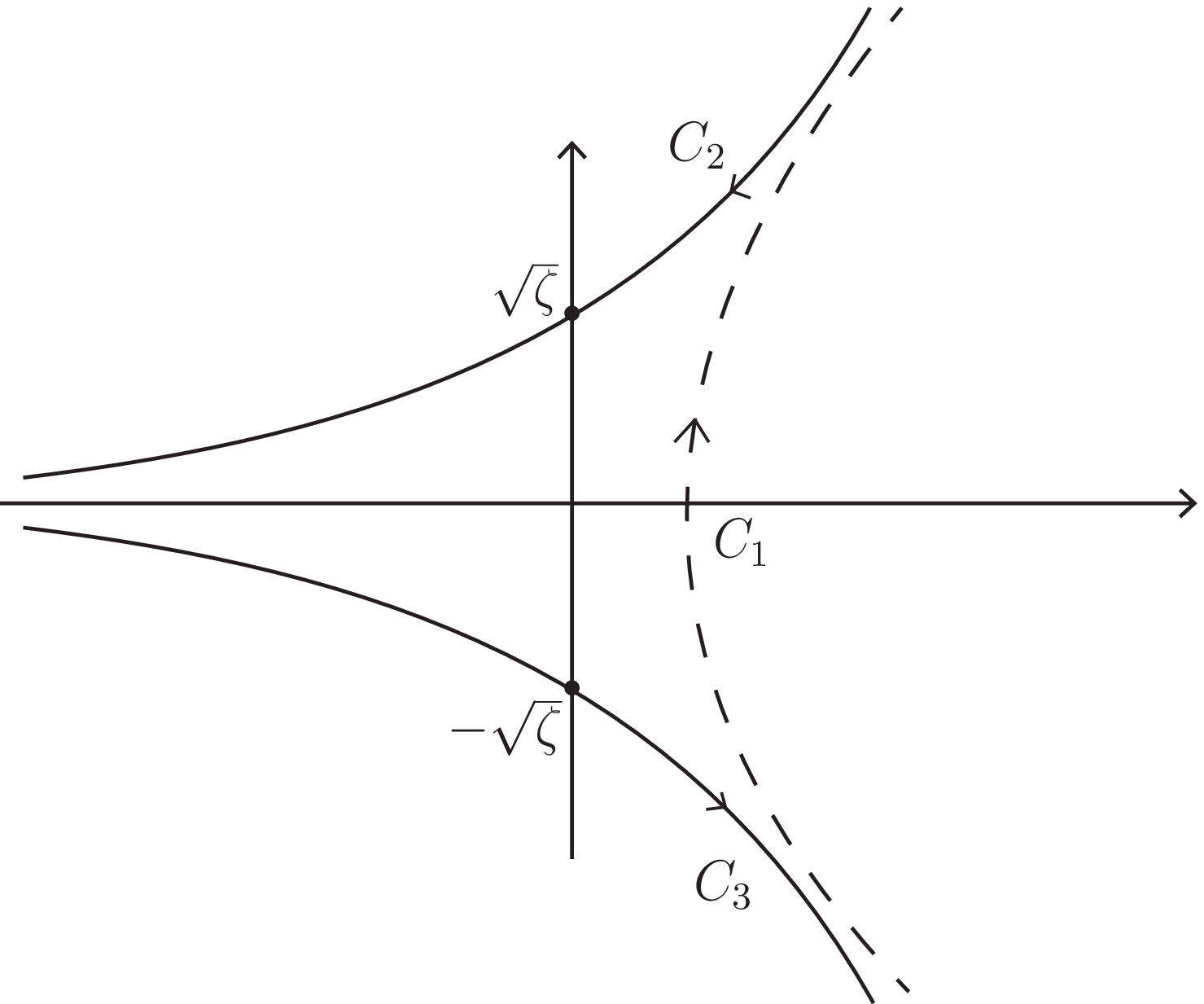}\\
  \caption{$C_2$ and $C_3$ $(\zeta<0)$.}\label{SDP of airy 2}
\end{minipage}\hfill
\begin{minipage}[t]{0.4\textwidth}
\centering
  \includegraphics[width=90pt,bb=180 0 350 170]{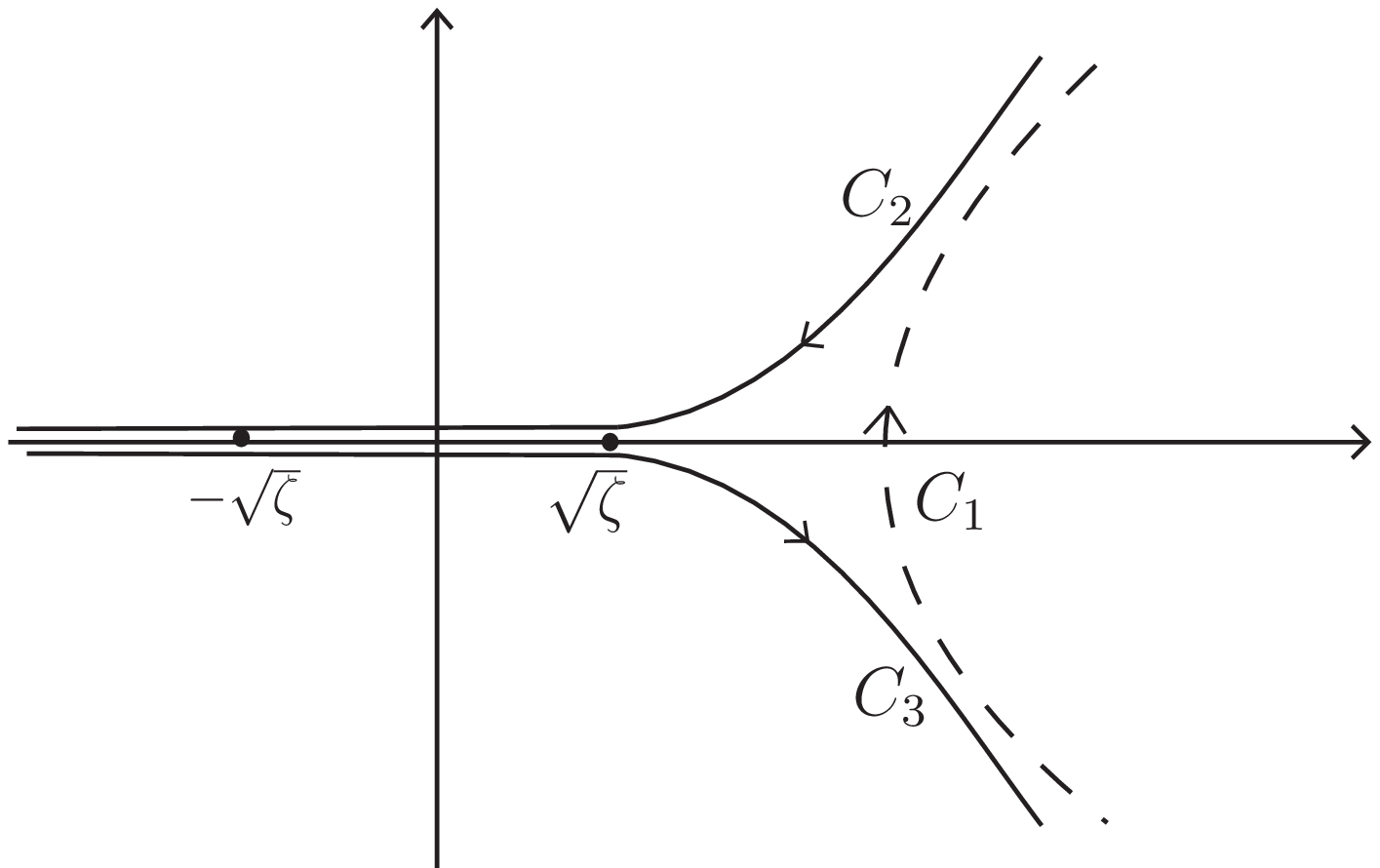}\\
  \caption{$C_2$ and $C_3$ $(\zeta>0)$.}\label{SDP of airy 1}
\end{minipage}\hfill
\end{figure}

%
\subsection{When $\zeta$ is bounded}
Following the standard integration-by-parts procedure
\cite[p.368]{R.Wong's book}, we define recursively
\begin{eqnarray}
  h_l(u,\tau) &=& a_l(\tau)+b_l(\tau)u-(u^2-\zeta)g_l(u,\tau),\label{h_l 1 airy}\\
  h_{l+1}(u,\tau)&=&\frac{\partial
g_l(u,\tau)}{\partial u}.\label{h_l 2 airy}
\end{eqnarray}
Furthermore, expand
\begin{equation}\label{airy coefficient expanded near tau=0}
       a_l(\tau)=\sum_{j=0}^{\infty} a_{l,j}\tau^j, \qquad       b_l(\tau)=\sum_{j=0}^{\infty}
       b_{l,j}\tau^j.
\end{equation}
From (\ref{after the airy transform}), (\ref{h_l 1 airy}), (\ref{h_l
2 airy}) and (\ref{airy coefficient expanded near tau=0}), it
follows
\begin{align}\label{expansion in Airy case 2}
 t_n(x,N+1)=& \frac{(-1)^{n}\Gamma(n+N+2)e^nb^{-n}e^{N
A}}{\Gamma(N-n+1)N^{n+1}}   \notag \\
&\times\left[\cos{aN\pi}\left(\text{Ai}(N^{2/3}\zeta)\sum_{l=0}^{p-1}\frac{c_l}{N^{l+1/3}}-\text{Ai}'(N^{2/3}\zeta)\sum_{l=0}^{p-1}\frac{d_l}{N^{l+2/3}}
\right)
\right. \\
&\left.\quad+\sin{aN\pi}\left(
\text{Bi}(N^{2/3}\zeta)\sum_{l=0}^{p-1}\frac{c_l}{N^{l+1/3}}-\text{Bi}'(N^{2/3}\zeta)\sum_{l=0}^{p-1}\frac{d_l}{N^{l+2/3}}\right)+\varepsilon_{p}\right],\notag
\end{align}
where
\begin{equation}\label{error term in expansion in case 2}
\begin{split}
\varepsilon_p=&\frac{\cos{(aN\pi)}N^{n+1}}{2\pi iN^p\Gamma{(n+1)}}
\int_{0}^{+\infty}\int_{C_1}h_p(u,\tau)e^{N
\left(\frac{1}{3}u^3-\zeta
u \right)}e^{N(b\ln \tau-\tau)}\mathrm{d}u\mathrm{d}\tau\\
&+\frac{\sin{(aN\pi)}N^{n+1}}{2\pi iN^p\Gamma{(n+1)}}
\int_{0}^{+\infty}\int_{C_3-C_2}h_p(u,\tau)e^{N
\left(\frac{1}{3}u^3-\zeta u\right)}e^{N(b\ln
\tau-\tau)}\mathrm{d}u\mathrm{d}\tau,
\end{split}
\end{equation}
and the coefficients are given by
\begin{equation}\label{airy coefficient in case 2}
\begin{split}
  c_l &=\frac{N^{n+1}}{\Gamma(n+1)}\int_0^\infty a_l(\tau)\tau^ne^{-N\tau}\mathrm{d}\tau\sim \sum_{m=0}^\infty a_{l,m}\frac{\Gamma(n+m+1)}{\Gamma(n+1)N^m}, \\
  d_l& =\frac{N^{n+1}}{\Gamma(n+1)}\int_0^\infty b_l(\tau)\tau^ne^{-N\tau}\mathrm{d}\tau\sim \sum_{m=0}^\infty b_{l,m}\frac{\Gamma(n+m+1)}{\Gamma(n+1)N^m}.
  \end{split}
\end{equation}
In the present case, $a/b^2\in[\delta, M]$ or, equivalently, $w_\pm$
are bounded. Thus $\zeta$ is also bounded in view of (\ref{airy
transformation mapping of saddle points b-0}). Therefore, it is easy
to prove that (\ref{expansion in Airy case 2}) is asymptotic; for
details, see  \cite[p.371-372]{R.Wong's book}.

\subsection{When $\zeta\rightarrow -\infty$}
When $a/b^2\rightarrow \infty$, by (\ref{saddle points of w}) the
saddle points $w_\pm$ approach $\frac{1}{2}\pm i\infty$,
respectively, along the line $\text{Re }w=\frac{1}{2}$. In view of
(\ref{airy transformation mapping of saddle points b-0}), this is
equivalent to saying that $\zeta\rightarrow -\infty$. To prove that
(\ref{expansion in Airy case 2}) is also asymptotic in this case, we
rewrite the expansion in (\ref{expansion in Airy case 2}) as
\begin{align}\label{expansion in Airy case 3}
t_n(x,N+1)=& \frac{(-1)^{n}\Gamma(n+N+2)e^nb^{-n}e^{N A}}{\Gamma(N-n+1)N^{n+1}}\notag \\
&\times\left[\cos{aN\pi}\left(\frac{\text{Ai}(N^{2/3}\zeta)}{N^{1/3}}\sum_{l=0}^{p-1}\frac{\widetilde{c}_l}{(a
N)^{l}}-\frac{\text{Ai}'(N^{2/3}\zeta)}{N^{2/3}}\sum_{l=0}^{p-1}\frac{\widetilde{d}_l}{(a
N)^{l}}\right)
\right.\\
&\quad\left.\ +\sin{aN\pi}\left(
\frac{\text{Bi}(N^{2/3}\zeta)}{N^{1/3}}\sum_{l=0}^{p-1}\frac{\widetilde{c}_l}{(a
N)^{l}}-\frac{\text{Bi}'(N^{2/3}\zeta)}{N^{2/3}}\sum_{l=0}^{p-1}\frac{\widetilde{d}_l}{(a
N)^{l}}\right)+\varepsilon_{p}\right],\notag
\end{align}
where $\varepsilon_{p}$ is the same as in (\ref{error term in
expansion in case 2}) and
\begin{equation}
\label{coefficient in case 3} \widetilde{c}_l=a^{l}c_l \quad \text{
and }\quad \widetilde{d}_l=a^{l}d_l.
\end{equation}
Thus, it is sufficient to first prove the boundedness of the
coefficients $\widetilde{c}_l$ and $\widetilde{d}_l$, and then
establish the asymptotic nature of the error term in (\ref{expansion
in Airy case 3}), namely, to prove that there exist positive
constants $A_p$, $A_p'$, $B_p$ and $B_p'$ such that
\begin{equation}\label{asymptotic nature of the error terms in airy case}
\begin{split}
    \left|\varepsilon_{p}^1\right|\leq &
    \frac{A_p}{(aN)^{p}}\frac{\widetilde{\text{Ai}}(N^{2/3}\zeta)}{N^{1/3}}+\frac{A_p'}{(aN)^{p}}\frac{\widetilde{\text{Ai}}'(N^{2/3}\zeta)}{N^{2/3}},\\
    \left|\varepsilon_{p}^2\right|\leq &
    \frac{B_p}{(aN)^{p}}\frac{\widetilde{\text{Bi}}(N^{2/3}\zeta)}{N^{1/3}}+\frac{B_p'}{(aN)^{p}}\frac{\widetilde{\text{Bi}}'(N^{2/3}\zeta)}{N^{2/3}},
\end{split}
\end{equation}
where
\begin{equation}\label{error term separated in case 3}
    \begin{split}
    \varepsilon_{p}^1=& \frac{\cos(aN\pi)N^{n+1}}{2\pi iN^p\Gamma(n+1)}
\int_{0}^{+\infty}\int_{C_1}h_p(u,\tau)e^{N
\left(\frac{1}{3}u^3-\zeta u \right)}e^{N(b\ln
\tau-\tau)}\mathrm{d}u\mathrm{d}\tau,\\
    \varepsilon_{p}^2=&\frac{\sin(aN\pi)N^{n+1}}{2\pi i N^p\Gamma(n+1)}
\int_{0}^{+\infty}\int_{C_3-C_2}h_p(u,\tau)e^{N
\left(\frac{1}{3}u^3-\zeta u\right)}e^{N(b\ln
\tau-\tau)}\mathrm{d}u\mathrm{d}\tau,
    \end{split}
\end{equation}
and
$\widetilde{\text{Ai}}(z)=\widetilde{\text{Bi}}(z)=\left[\text{Ai}^2(z)+\text{Bi}^2(z)\right]^{1/2}$
for $z<0$. Note that in the present case, $\zeta<0$.\\

From (\ref{dw over du case 2 airy}), we recall that
$\mathrm{d}w/\mathrm{d}u$ depends on $\zeta$. The value of $\zeta$
is obtained by solving the two equations gotten from (\ref{airy
transformation in case 2 b-0}) with $w$ and $u$ replaced
respectively by $w_\pm$ and $\pm \sqrt{\zeta}$; see
\cite[p.367]{R.Wong's book}. Thus, $\zeta$ and hence
$\mathrm{d}w/\mathrm{d}u$ both depend on the parameters $a$ and $b$
in (\ref{a and b}). As $a$ and $b$ approach zero, $\zeta$ may tend
to infinity. When the saddle points $u_\pm=\pm\sqrt{\zeta}$ are
bounded, the mapping $w\rightarrow u$ defined by the cubic
transformation (\ref{airy transformation in case 2 b-0}) is
analytic, and its derivative $\mathrm{d}w/\mathrm{d}u$ is bounded.
However, in the present case, the saddle points
$u_\pm=\pm\sqrt{-\zeta}i$ go to infinity as $\zeta\rightarrow
-\infty$. Hence, the coefficients $a_l(\tau)$, $b_l(\tau)$ and the
function $h_l(u,\tau)$ given recursively in (\ref{h_l 1 airy}) and
(\ref{h_l 2 airy}) may all blow up as $\zeta\rightarrow -\infty$,
where $h_0(u,\tau)$ is defined by (\ref{h0 in airy case}). Since the
coefficients $c_l$ and $d_l$ in (\ref{expansion in Airy case 2}) are
related to $a_l(\tau)$ and $b_l(\tau)$ via (\ref{airy coefficient in
case 2}) and (\ref{airy coefficient expanded near tau=0}), to prove
that the expansion in (\ref{expansion in Airy case 3}) is asymptotic
when $\zeta\rightarrow -\infty$, we must first give estimates for
the coefficient functions $a_l(\tau)$ and $b_l(\tau)$. To this end,
we shall adopt a method introduced by Olde Daalhuis and Temme
\cite{Olde Daalhuis}. To begin with, we define
\begin{equation}\label{A0 and B0}
    A_0(u,\zeta)=\frac{u}{u^2-\zeta},\qquad
    B_0(u,\zeta)=\frac{1}{u^2-\zeta}.
\end{equation}
Using (\ref{h_l 1 airy}) and the Cauchy residue theorem, it can be
verified that
\begin{equation}\label{a0' and b0'}
    \begin{split}
    a_0(\tau)=&\frac{1}{2\pi
    i}\int_{\Gamma}h_0(u,\tau)A_0(u,\zeta)\mathrm{d}u,\\
    b_0(\tau)=&\frac{1}{2\pi
    i}\int_{\Gamma}h_0(u,\tau)B_0(u,\zeta)\mathrm{d}u,
    \end{split}
\end{equation}
where $\Gamma$ is the contour consisting of two circles, centering
at $\pm\sqrt{-\zeta} i$, both with radius $R$, where $R$ could be as
large as possible until the circles reach the singularities of
$h_0(u,\tau)$ in the $u$-plane.\\

We note that $w=w_\pm$ are removable singularities of
$\mathrm{d}w/\mathrm{d}u$. However, $\mathrm{d}w/\mathrm{d}u$ blows
up at the points $w=w_\pm e^{2k\pi i}$, where $k\neq 0$ is any
integer; see (\ref{dw over du case 2 airy}). To find their image
points in the $u$-plane under the mapping (\ref{airy transformation
in case 2 b-0}), we take $w_k=w_+e^{2k\pi i}$ as an example and
denote its image point by $u_k$. The image point of $w_-e^{2k\pi i}$
can be treated in a similar manner. By (\ref{airy transformation in
case 2 b-0}), we have
\begin{equation}\label{singularities mapping}
    \overline{f}(t_0(w_k),w_k)=\frac{1}{3}u_k^3-\zeta u_k+A.
\end{equation}
From (\ref{airy transformation in case 2 b-0}) and (\ref{airy
transformation mapping of saddle points b-0}), we also have
\begin{equation}\label{singularities mapping 2}
    \overline{f}(t_0(w_+),w_+)=\frac{1}{3}u_+^3-\zeta u_++A.
\end{equation}
Subtracting (\ref{singularities mapping 2}) from (\ref{singularities
mapping}) gives
\begin{equation}\label{temp temp}
    2k\pi a i=i(-\zeta)^{3/2}\left(-\frac{1}{3}z_k^3-z_k^2\right).
\end{equation}
Since $0<a\leq \frac{1}{2}$ and $-\zeta \gg O(1)$, from (\ref{temp
temp}) it follows that $|z_k|\sim
\sqrt{2k\pi}a^{1/2}(-\zeta)^{-3/4}$. Therefore, there exists a
constant $0<c_0<1$, independent of $a$, $b$ and $\zeta$, such that
the interior of the two circles with centres at $u=\pm
\sqrt{-\zeta}i$ and radius $R=c_0 \sqrt{a} (-\zeta)^{-1/4}$ is free
of the singularities of $h_0(u,\tau)$ in the $u$-plane. Since
$h_0(u,\tau)$ is now analytic inside the contour $\Gamma$, there
exists a constant $c_h$ such that
\begin{equation}\label{estimate of h0}
    \left|h_0(u,\tau)\right|\leq c_h h(\tau)
\end{equation}
for $u$ in the domain enclosed by the contour $\Gamma$, where
$h(\tau)$ denotes the maximum of the two functions
$\left|h_0(\pm\sqrt{-\zeta}i,\tau)\right|$. Note that as functions
of $\tau$, $h_0(\pm\sqrt{-\zeta}i,\tau)$ are analytic in
the neighbourhood of steepest descent path in the $\tau$-plane.\\

We further introduce rational functions $A_k$ and $B_k$,
$k=0,1,2...$, defined recursively by
\begin{equation}\label{}
    \begin{split}
    A_{k+1}(u,\zeta)=&\frac{1}{u^2-\zeta}\frac{\mathrm{d}}{\mathrm{d}u}A_k(u,\zeta),\\
    B_{k+1}(u,\zeta)=&\frac{1}{u^2-\zeta}\frac{\mathrm{d}}{\mathrm{d}u}B_k(u,\zeta),
    \end{split}
\end{equation}
where $A_0$ and $B_0$ are given in (\ref{A0 and B0}). By induction,
we can show that $A_k$ and $B_k$ are expressible as
\begin{equation}\label{An and Bn summation}
    \begin{split}
    A_{k}(u,\zeta)=&\sum_{i=0}^{\left[(k+1)/2\right]}\frac{c^1_{k,i} u^{k+1-2i}}{\left(u^2-\zeta\right)^{2k+1-i}},\\
    B_{k}(u,\zeta)=&\sum_{i=0}^{\left[k/2\right]}\frac{c^2_{k,i} u^{k-2i}}{\left(u^2-\zeta\right)^{2k-i}},
    \end{split}
\end{equation}
where $c^1_{k,i}$ and $c^2_{k,i}$ are constants independent of $u$
and $\zeta$. As in (\ref{a0' and b0'}), by Cauchy's theorem, we have
from equations (\ref{h_l 1 airy}) and (\ref{h_l 2 airy})
\begin{equation}\label{IR of a-k}
    \begin{split}
    a_k(\tau)=&\frac{1}{2\pi
    i}\int_{\Gamma}h_k(u,\tau)A_0(u,\zeta)\mathrm{d}u\\
    =&\frac{1}{2\pi
    i}\int_{\Gamma}h_{k-1}(u,\tau)A_1(u,\zeta)\mathrm{d}u-\frac{1}{2\pi
    i}\int_{\Gamma}\left(a_{k-1}(\tau)+b_{k-1}(\tau)u\right)A_1(u,\zeta)\mathrm{d}u\\
    =&\frac{1}{2\pi
    i}\int_{\Gamma}h_{k-1}(u,\tau)A_1(u,\zeta)\mathrm{d}u\\
    \vdots &\\
    =&\frac{1}{2\pi
    i}\int_{\Gamma}h_0(u,\tau)A_k(u,\zeta)\mathrm{d}u,
    \end{split}
\end{equation}
where we have used integration-by-parts to derive the second
equality. The second term in the second equality vanishes because
$\left(a_{k-1}(\tau)+b_{k-1}(\tau)u\right)A_1(u,\zeta)$ is
$O\left(u^{-2}\right)$ as $|u|\rightarrow \infty$ and all poles of
that function lie inside $\Gamma$; see (\ref{An and Bn summation}).
Similarly, we also have
\begin{equation}\label{IR of b-k}
    b_k(\tau)=\frac{1}{2\pi
    i}\int_{\Gamma}h_0(u,\tau)B_k(u,\zeta)\mathrm{d}u.
\end{equation}
Using (\ref{An and Bn summation}), it is easy to obtain the
estimates
\begin{equation}\label{estimate of An and Bn}
    \left|A_k(u,\zeta)\right|\leq C_k a^{-k-1/2} (-\zeta)^{1/4}\quad\text{and}\quad  \left|B_k(u,\zeta)\right|\leq C_k
    a^{-k},
\end{equation}
for $u$ on and inside the contour $\Gamma$ and $-\zeta\rightarrow
\infty$. Here and thereafter, $C_k$ is used as a generic symbol for
constants independent of $u$, $\zeta$, $a$ and $b$. Substituting
(\ref{estimate of An and Bn}) into (\ref{IR of a-k}) and (\ref{IR of
b-k}) gives
\begin{equation}\label{}
    \left|a_k(\tau)\right|\leq C_k h(\tau)a^{-k},
\end{equation}
and
\begin{equation}\label{}
    \left|b_k(\tau)\right|\leq C_k h(\tau)a^{-k+1/2}(-\zeta)^{-1/4}.
\end{equation}
Therefore, $a^la_l(\tau)$ and $a^lb_l(\tau)$ are both bounded for
$\tau$ in the neighbourhood of steepest descent path in the
$\tau$-plane, and of course for $\tau$ in the neighbourhood of
$\tau=0$. Thus we have the boundedness of the coefficients
$\widetilde{c}_l$ and $\widetilde{d}_l$.\\

To estimate $h_k(u,\tau)$, we use the rational functions
$R_k(u,w,\zeta)$, $k=0,1,2...$, defined recursively by
\begin{equation}\label{R-k definition}
    \begin{split}
        R_0(u,w,\zeta)=&\frac{1}{u-w},\\
        R_{k+1}(u,w,\zeta)=&\frac{1}{u^2-\zeta}\frac{\mathrm{d}}{\mathrm{d}u}R_k(u,w,\zeta).
    \end{split}
\end{equation}
These functions were also introduced by Olde Daalhuis and Temme
\cite{Olde Daalhuis}. They showed by induction that $R_k(u,w,\zeta)$
can be written as
\begin{equation}\label{Rn summation form}
    R_k(u,w,\zeta)=\sum_{i=0}^{k-1}\sum_{j=0}^{\min\{i,k-1-i\}}\frac{C_{ij}u^{i-j}}{(u-w)^{k+1-i-j}\left(u^2-\zeta\right)^{k+i}},\qquad\qquad
    k=1,2...,
\end{equation}
where $C_{ij}$ do not depend on $u$, $w$ and $\zeta$. Similar to
(\ref{IR of a-k}), we have
\begin{equation}\label{IR of h-k}
    h_k(w, \tau)=\frac{1}{2\pi
    i}\int_{\Gamma}h_0(u, \tau)R_k(u,w,\zeta)\mathrm{d}u,
\end{equation}
where $\Gamma$ is the same contour used in (\ref{IR of a-k}) and $w$
lies inside two disks centered at $\pm \sqrt{-\zeta} i$ and with
radius $\frac{1}{2}c_0 \sqrt{a} (-\zeta)^{-1/4}$. It is easy to
verify from (\ref{Rn summation form}) that
\begin{equation}\label{estimate of R-k}
    \left|R_k(u,w,\zeta)\right|\leq C_ka^{-k-1/2}(-\zeta)^{1/4},
\end{equation}
and from (\ref{IR of h-k}) that
\begin{equation}\label{estimate of h-k}
    \left|h_k(w,\tau)\right|\leq
    C_k a^{-k}h(\tau).
\end{equation}
Substituting (\ref{estimate of h-k}) into (\ref{error term separated
in case 3}) gives (\ref{asymptotic nature of the error terms in airy
case}); for details, see \cite[p.311-312]{Olde Daalhuis}. Note that
to make the expansion (\ref{expansion in Airy case 3}) asymptotic,
we require $x=aN$ to be large.

\section{BESSEL-TYPE EXPANSION} In the case of Hahn Polynomials $Q_n(x;\alpha,\beta,N)$ given
in (\ref{definition of Q}), Sharapodinov \cite{Sharapodinov} has
given an asymptotic formula involving Jacobi polynomials when the
parameters satisfy $\alpha$, $\beta\geq \frac{1}{2}$ and $2\leq
n\leq c\sqrt{N}$, where $c$ is a positive constant. The values of
the variable $x$ are also required to be large; more precisely,
$x\geq \varepsilon N$ and $\varepsilon>0$ is a small number.
Although discrete Chebyshev polynomials $t_n(x,N)$ given in
(\ref{relation between DCP and Hahn}) is a special case of the Hahn
polynomials, the values of the parameters are $\alpha=\beta=0$; that
is, Sharapodinov's result does not include our case. However, since
the leading term in the uniform asymptotic expansion of the Jacobi
polynomials is a Bessel function (see \cite[p.451]{handbook}), the
work of Sharapodinov did inspire us to look for an asymptotic
expansion for $t_n(x, N+1)$ involving Bessel functions, when the
parameters $a$ and $b$ in (\ref{a and b}) satisfy $a/b^2\rightarrow
\infty$ and the variable $x$ is large. Our method differs completely
from that of
Sharapodinov.\\

Returning to (\ref{IR of DCP for x>0}), and making the change of
variable $v=1/w$, we have
\begin{equation}\label{IR of DCP for x>0 case 3}
t_n(x,N+1)=\frac{(-1)^{n}}{2\pi
i}\frac{\Gamma(n+N+2)}{\Gamma(n+1)\Gamma(N-n+1)}\int_0^1\int_{\gamma_3}\frac{1}{v(1-v)}
e^{N \widehat{f}(t,v)}\mathrm{d}v\mathrm{d}t,
\end{equation}
where
\begin{equation}\label{phase function f for x>0 2}
\widehat{f}(t,v)=b \ln(1-t)+(1-b)\ln t-a\ln (1-v)-b \ln v+b
\ln[t+v-1];
\end{equation}
the curve $\gamma_3$ starts at $v=+\infty$, runs along the lower
edge of the positive real line towards $v=1$, encircles the point
$v=1$ in the clockwise direction and returns to $+\infty$ along the
upper edge of the positive real line. Note that since $e^{Nb\ln
v}=v^n$, there is no need to have a cut from the origin to
infinity.\\

The saddle point of $\widehat{f}(t,v)$ in the $t$-plane is given by
\begin{equation}\label{saddle point of t in case 3}
    t_0(v)=\frac{2-v+\sqrt{4b^2-4b^2 v+v^2}}{2(1+b)};
\end{equation}
cf.(\ref{saddle points for t}), where the $v$-plane is cut along two
line segments joining $0$ to the two conjugate points
$2b^2\pm2b\sqrt{1-b^2}i$, and the branch of the square root is
chosen so that $\sqrt{4b^2-4b^2 v+v^2}\sim v$ as $v\rightarrow
\infty$. From (\ref{saddle point of t in case 3}), we have
\begin{equation}\label{t0v as v large or fixed}
\begin{split}
    t_0(v)&=1+O(b)\qquad\hspace{-0.2cm} \text{as }|v|\sim cb,\\
    t_0(v)&\sim 1-b\qquad\hspace{0.5cm}\text{as    }v\gg b,
\end{split}
\end{equation}
where $c$ is a positive constant, and for $v\ll b$,
\begin{equation}\label{t0v as v small}
\begin{split}
t_0(v)&\sim 1-\frac{1}{2}v\qquad \text{as }v\rightarrow 0^+,\\
t_0(v)&\rightarrow 1-2b \qquad\text{as } v\rightarrow 0^-,
\end{split}
\end{equation}
where $0^+$ and $0^-$ mean limits approaching $0$ from the
right-hand side of the cut and the left-hand side of the cut,
respectively. Moreover, easy calculation shows that for any $v$,
$t_0(v)$ is not a real number on the cut $(1,\infty)$ in the
$t$-plane. Following the same argument given prior to (\ref{mapping
t to tau case 1}), we introduce the mapping
\begin{equation}\label{mapping t to tau case 3}
    \widehat{f}(t,v)=b \ln \tau-\tau+A
\end{equation}
with the correspondence between the saddle points $t=t_0(v)$ and
$\tau=b$ given by
\begin{equation}\label{correspondence between the saddle points in case 3}
    \tau(t_0(v))=b.
\end{equation}
Coupling (\ref{mapping t to tau case 3}) and (\ref{correspondence
between the saddle points in case 3}) yields
\begin{equation}\label{constant A in case 3}
    A=\widehat{f}(t_0(v),v)-b\ln b+b.
\end{equation}
The zeros of $\partial \widehat{f}(t,v)/\partial t=0$ are given by
\begin{equation}\label{t0pm}
    t_0^{\pm}(v)=\frac{2-v\pm \sqrt{4b^2-4b^2 v+v^2}}{2(1+b)},
\end{equation}
where $t_0^+(v)$ is the relavant saddle point $t_0(v)$ given in
$(\ref{saddle point of t in case 3})$. By straightforward
calculation, we have
\begin{equation}\label{dt over dtau b-0}
        \begin{split}
            \frac{\mathrm{d}t}{\mathrm{d}\tau}& =\frac{b/\tau-1}{\partial \widehat{f}(t,v)/\partial
            t}=\frac{(\tau-b) t(1-t)(t+v-1)}{\tau(1+b)(t-t^+_0(v))(t-t^-_0(v))},\quad\quad \tau\neq
            b,\\
            \frac{\mathrm{d}t}{\mathrm{d}\tau}&=\left\{\frac{t^+_0(v)(1-t^+_0(v))(t^+_0(v)+v-1)}{b\sqrt{4b^2-4b^2v+v^2}}\right\}^{1/2}, \hspace{1.6cm}
            \tau=b.
        \end{split}
\end{equation}
Note that we have $\tau=0$ when $t=1$, and $\tau=+\infty$ when
$t=0$. For any fixed $v$, we can deform the original interval of
integration $0\leq t\leq 1$ into a steepest descent path $\Gamma_v$,
passing through $t_0(v)$. Also note that $t_0^-(v)$ is not on
$\Gamma_v$, unless $0<v<1$. Moreover, if $0<v<1$, then
$0<t_0^-(v)<t_0^+(v)<1$ and $\Gamma_v$ is the real interval $[0,1]$.
In our case, there is only one point on the path $\gamma_3$ in the
$v$-plane (see (\ref{IR of DCP for x>0 case 3})), where it crosses
the real line. Let us denote this point by $v_0$. When $v=v_0$, the
mapping which we have introduced in (\ref{mapping t to tau case 3})
becomes singular at the point $t_0^-(v)$, since
$\mathrm{d}t/\mathrm{d}\tau$ in (\ref{dt over dtau b-0}) blows up.
However for this particular case, we only need to slightly modify
the path by replacing part of the original path near this point by a
small half circle as shown in Figure \ref{t SDP with dentation}.
\begin{figure}[!h]
\centering
  \includegraphics[scale=0.7, bb=0 80 353 187]{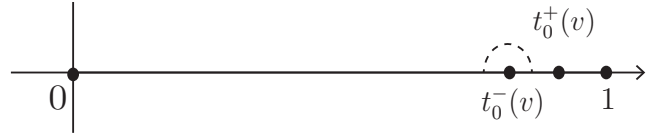}\\
  \caption{The indentation of the integration path in the $t$-plane for $0<v<1$.}\label{t SDP with dentation}
\end{figure}
The corresponding integration path in the $\tau$-plane following the
mapping (\ref{mapping t to tau case 3}) also needs to be modified.
However, this small modification on the integration path will not
affect the following argument and calculation. With this in mind, we
will simply ignore this particular case, and proceed with the
assumption that the integration path in the $\tau$-plane is always
$[0,+\infty)$ for all $w$, and that $\mathrm{d}t/\mathrm{d}\tau$
will not blow up in the neighbourhood of the path.\\

Thus, from (\ref{IR of DCP for x>0 case 3}) and (\ref{mapping t to
tau case 3}), we have
\begin{equation}\label{after t mapping in case 3}
\begin{split}
t_n(x,N+1)=& \frac{(-1)^{n+1}}{2\pi
i}\frac{\Gamma(n+N+2)e^nb^{-n}}{\Gamma(n+1)\Gamma(N-n+1)}\\
&\times \int_{0}^{+\infty}\int_{\gamma_3}\frac{1}{v(1-v)}e^{N
\widehat{f}({t_0(v)},v)}e^{N(b\ln
\tau-\tau)}\frac{\mathrm{d}t}{\mathrm{d}
\tau}\mathrm{d}v\mathrm{d}\tau.
\end{split}
\end{equation}
Here, we rewrite the phase function $\widehat{f}(t_0(v),v)$ as
\begin{equation}\label{f widehat rearranged}
    \widehat{f}(t_0(v),v)=(1-b)\ln t_0(v)+b\ln (1-t_0(v))-a\ln (1-v)-b\ln(\frac{v}{t_0(v)+v-1}).
\end{equation}
Recalling the statements following (\ref{phase function f for x>0
2}) and (\ref{saddle point of t in case 3}), we know that there are
only two cuts in the $v$-plane: one along the infinite interval
$[1,+\infty)$ and the other along the bent line joining the
conjugate points $2b^2\pm 2b\sqrt{1-b^2}i$ and passing through the
origin. To find the saddle points of $\widehat{f}(t_0(v),v)$, we set
\begin{equation*}
\frac{\partial \widehat{f}(t_0(v),v)}{\partial v}=0,
\end{equation*}
and obtain
\begin{equation}\label{saddle points of v}
    v_\pm=\frac{b^2\mp b i\sqrt{4a-4a^2-b^2}}{2a (1-a)}.
\end{equation}
Since (\ref{IR of DCP for x>0 case 3}) is obtained from (\ref{IR of
DCP for x>0}) by making the change of variable $v=1/w$, (\ref{saddle
points of v}) can also be derived from (\ref{saddle points of w}).
Note that in this case, $b^2/a\rightarrow 0$. Furthermore, since
$a<\frac{1}{2}$, the quantity inside the square root is positive.
Hence, $v_\pm$ are distinct, and approach $v=0$.\\

Define
\begin{equation*}
    g(u)=m(u-\frac{1}{u}),
\end{equation*}
where $m>0$ is some constant to be determined. The saddle points of
$g(u)$ are
\begin{equation}\label{saddle points of u in case 3}
    u_\pm=\pm i.
\end{equation}
Make the transformation
\begin{equation}\label{mapping in case 3}
    \begin{split}
    \widehat{f}(t_0(v),v)&=g(u)+\gamma\\
    &=m(u-\frac{1}{u})+\gamma
    \end{split}
\end{equation}
with
\begin{equation}\label{the relation between v+- and u+- case 3}
    u(v_+)=u_-,\quad\quad\quad\quad u(v_-)=u_+.
\end{equation}
Note the fact that $t_0(v_\pm)=t_\pm$, where $t_\pm$ are given in
(\ref{sets of saddle points 1}) and (\ref{sets of saddle points 2}).
This can be seen from (\ref{temptemp123}) and the change of variable
$v=1/w$ that we have made. Thus, substituting (\ref{the relation
between v+- and u+- case 3}) into (\ref{mapping in case 3}) gives
\begin{equation}\label{gamma in the mapping case 3}
\begin{split}
    \gamma&=\frac{1}{2}\ln \frac{1-b}{1+b}+\frac{1}{2}b\ln
    \frac{b^2}{1-b^2}\\
    &=b\ln b-b+O(b^3)
    \end{split}
\end{equation}
and
\begin{equation}
\label{asymptotic behaviour of m}
\begin{split}
 m=-\frac{1}{2}\biggr\{&(1-b)\arctan\frac{b\sqrt{4a-4a^2-b^2}}{2-2a-b^2}-a\arctan\frac{b\sqrt{4a-4a^2-b^2}}{2a-2a^2-b^2}\\
    &\left.-2b\arctan\frac{\sqrt{4a-4a^2-b^2}}{2+b-2a}\right\}\\
 \sim\  b&\arctan
\sqrt{\frac{a}{1-a}},
\end{split}
\end{equation}
when $b^2/a\rightarrow 0$ and $b\rightarrow 0$. Moreover, from
(\ref{mapping in case 3}) we have
\begin{equation}\label{partial v over partial u in case 3}
        \begin{split}
            &\frac{\mathrm{d}v}{\mathrm{d}u}=\frac{m(u-i)(u+i)v(1-v)\left[(1-2a)v+\sqrt{4b^2-4b^2 v+v^2}\right]}{2a(1-a)u^2(v-v_+)(v-v_-)},\quad\quad u\neq
            u_\pm
            ,\\
            &\left.\frac{\mathrm{d}v}{\mathrm{d}
            u}\right|_{u=u_\pm}=\left\{\frac{-2m(v_\mp)^2(1-v_\mp)(1-2a)}{b\sqrt{4a-4a^2-b^2}}\right\}^{1/2},\qquad\qquad
            u=u_\pm.
        \end{split}
\end{equation}
Here we have made use of the equality
$\sqrt{4b^2-4b^2v_\mp+v_\mp^2}=(1-2a)v_\mp$.\\

Coupling (\ref{after t mapping in case 3}) and (\ref{mapping in case
3}), the integral representation of $t_n(x, N+1)$ becomes
\begin{equation}\label{after the t and u mapping in case 3}
\begin{split}
t_n(x,N+1)=& \frac{(-1)^{n+1}}{2\pi
i}\frac{\Gamma(n+N+2)e^nb^{-n}e^{N\gamma}}{\Gamma(n+1)\Gamma(N-n+1)}\\
&\times
\int_{0}^{+\infty}\int_{\widehat{\gamma}}\frac{h(u,\tau)}{u}e^{N
\left[m\left(u-\frac{1}{u}\right)\right]}e^{N(b\ln
\tau-\tau)}\mathrm{d}u\mathrm{d}\tau,
\end{split}
\end{equation}
where
\begin{equation}\label{h function after the mapping in case 3}
    h(u,\tau)=\frac{u}{v(1-v)}\frac{\mathrm{d}t}{\mathrm{d}
\tau}\frac{\mathrm{d}v}{\mathrm{d} u},
\end{equation}
and the contour $\widehat{\gamma}$ starts from $-\infty$, encircles
the point $u=0$ in the counterclockwise direction and returns to
$-\infty$. For $l=0,1,2,\cdot\cdot\cdot$, we define recursively
\begin{eqnarray}
  h_l(u,\tau) &=& a_l(\tau)+\frac{b_l(\tau)}{u}+\left(1+\frac{1}{u^2}\right)g_l(u,\tau),\label{h_l in case 3}\\
  h_{l+1}(u,\tau)&=&-\frac{u}{m}\frac{\mathrm{d}}{\mathrm{d}
u}\left\{\frac{g_l(u,\tau)}{u}\right\},\label{hl 2 in case 3}
\end{eqnarray}
where $h_0(u,\tau)=h(u,\tau)$ given in (\ref{h function after the
mapping in case 3}). Furthermore, we expand $a_l(\tau)$ and
$b_l(\tau)$ at $\tau=0$, and write
\begin{equation}\label{expansion of coeffiencet an and bn in case 3}
       a_l(\tau)=\sum_{j=0}^{\infty} a_{l,j}\tau^j, \qquad       b_l(\tau)=\sum_{j=0}^{\infty}
       b_{l,j}\tau^j.
\end{equation}
It is easy to see that
\begin{equation}\label{al bl and hl}
    \begin{split}
    a_l(\tau)&=\frac{1}{2}\left[h_l(i,\tau)+h_l(-i,\tau)\right],\\
     b_l(\tau)&=\frac{i}{2}\left[h_l(i,\tau)-h_l(-i,\tau)\right].
    \end{split}
\end{equation}
From (\ref{after the t and u mapping in case 3}), (\ref{h_l in case
3}) and (\ref{hl 2 in case 3}), we have
\begin{align}\label{substitute in in case 3 temp}
  t_n(x,N+1)=& \frac{(-1)^{n+1}}{2\pi
i}\frac{\Gamma(n+N+2)e^nb^{-n}e^{N\gamma}}{\Gamma(n+1)\Gamma(N-n+1)} \notag\\
   &\times
\int_{0}^{+\infty}\int_{-\infty}^{\left(0^+\right)}\left[a_0(\tau)+\frac{b_0(\tau)}{u}+\left(1+\frac{1}{u^2}\right)g_0(u,\tau)\right]\frac{e^{N
m\left(u-\frac{1}{u}\right)}}{u}\tau ^n
e^{-N\tau}\mathrm{d}u\mathrm{d}\tau \\
=&\frac{(-1)^{n+1}\Gamma(n+N+2)e^nb^{-n}e^{N\gamma}}{\Gamma(N-n+1)N^{n+1}}\biggr[c_0J_0(2Nm)+d_0
J_1(2Nm)+\varepsilon_1^+\biggr],\notag
\end{align}
where
\begin{equation}\label{temp c0 and d0}
\begin{split}
  c_0 &=\frac{N^{n+1}}{\Gamma(n+1)}\int_0^\infty a_0(\tau)\tau^ne^{-N\tau}\mathrm{d}\tau\sim \sum_{m=0}^\infty a_{0,m}\frac{\Gamma(n+m+1)}{\Gamma(n+1)N^m}, \\
  d_0 & =\frac{N^{n+1}}{\Gamma(n+1)}\int_0^\infty b_0(\tau)\tau^ne^{-N\tau}\mathrm{d}\tau\sim \sum_{m=0}^\infty b_{0,m}\frac{\Gamma(n+m+1)}{\Gamma(n+1)N^m}
\end{split}
\end{equation}
and
\begin{equation}\label{epsilon 1 in case 3}
\begin{split}
  \varepsilon_1^+=\frac{N^{n+1}}{2\Gamma(n+1)\pi i}\int_{0}^{+\infty}\int_{-\infty}^{\left(0^+\right)}\left(1+\frac{1}{u^2}\right)\frac{g_0(u,\tau)}{u}e^{N m\left(u-\frac{1}{u}\right)}\tau ^n
e^{-N\tau}\mathrm{d}u\mathrm{d}\tau.
\end{split}
\end{equation}
In (\ref{substitute in in case 3 temp}), we have made use of the
integral representations of Bessel function and Gamma function
\begin{equation*}
    J_v(z)=\frac{1}{2\pi
    i}\int_{-\infty}^{\left(0^+\right)}u^{-v-1}\exp\left\{\frac{z}{2}\left(u-\frac{1}{u}\right)\right\}\mathrm{d}u
\end{equation*}
and
\begin{equation*}
    \Gamma(n+1)=N^{n+1}\int_0^{+\infty}e^{-N\tau}\tau^n\mathrm{d}\tau;
\end{equation*}
see \cite[(10.9.19) and (5.9.1)]{handbook}. Using (\ref{hl 2 in case
3}) and integration by parts, we can rewrite $\varepsilon_1^+$ in
(\ref{epsilon 1 in case 3}) as
\begin{equation}
\begin{split}
  \varepsilon_1^+=\frac{N^{n}}{2\Gamma(n+1)\pi i}\int_{0}^{+\infty}\int_{-\infty}^{\left(0^+\right)}\frac{h_1(u,\tau)}{u}e^{N
m\left(u-\frac{1}{u}\right)}\tau ^n
e^{-N\tau}\mathrm{d}u\mathrm{d}\tau.
\end{split}
\end{equation}
Repeating the procedure above, we obtain
\begin{equation}\label{result a>0 in case 3}
\begin{split}
  t_n(x,N+1) = \frac{(-1)^{n+1}\Gamma(n+N+2)e^n b^{-n}e^{N\gamma}}{\Gamma(N-n+1)N^{n+1}}\biggr[&J_0(2Nm)\sum_{l=0}^{p-1}\frac{c_l}{N^{l}}\\
   & \left. +J_1(2Nm)\sum_{l=0}^{p-1}\frac{d_l}{N^{l}}+\varepsilon_p^+\right],
\end{split}
\end{equation}
where
\begin{eqnarray}
  c_l &=&\frac{N^{n+1}}{\Gamma(n+1)}\int_0^\infty a_l(\tau)\tau^ne^{-N\tau}\mathrm{d}\tau\sim \sum_{m=0}^\infty a_{l,m}\frac{\Gamma(n+m+1)}{\Gamma(n+1)N^m} \label{cl when a>0 in case 3},\\
  d_l& =&\frac{N^{n+1}}{\Gamma(n+1)}\int_0^\infty b_l(\tau)\tau^ne^{-N\tau}\mathrm{d}\tau\sim \sum_{m=0}^\infty b_{l,m}\frac{\Gamma(n+m+1)}{\Gamma(n+1)N^m} \label{d_l when a>0 in case 3},\\
\varepsilon_p^+&= &\frac{N^{n+1-p}}{2\Gamma(n+1)\pi
i}\int_{0}^{+\infty}\int_{-\infty}^{\left(0^+\right)}\frac{h_p(u,\tau)}{u}e^{N
m\left(u-\frac{1}{u}\right)}\tau ^n
e^{-N\tau}\mathrm{d}u\mathrm{d}\tau\label{epsilon p in case 3}.
\end{eqnarray}
Since $h_0(u,\tau)$ involves ${\mathrm{d}v}/{\mathrm{d}u}$ and
${\mathrm{d}v}/{\mathrm{d}u}$ depends on the parameters $a$ and $b$
in (\ref{a and b}), the coefficients $c_l$ and $d_l$ in (\ref{cl
when a>0 in case 3}) and (\ref{d_l when a>0 in case 3}) also depend
on $a$ and $b$. For an estimate on these coefficients, see (\ref{cl
dl estimate in bessel}) below.\\

To facilitate the application of expansion (\ref{result a>0 in case
3}), we recall that the constants $\gamma$ and $m$ are explicitly
given in (\ref{gamma in the mapping case 3}) and (\ref{asymptotic
behaviour of m}), and note that the leading coefficient can be
(asymptotically) calculated by using (\ref{expansion of coeffiencet
an and bn in case 3}) and (\ref{al bl and hl}). Indeed, we have
\begin{equation*}
  c_0 \sim a_{0,0}=\frac{1}{2}\left[h_0(i,0)+h(-i,0)\right]
\end{equation*}
and
\begin{equation*}
  d_0 \sim b_{0,0}=\frac{i}{2}\left[h_0(i,0)-h(-i,0)\right]
\end{equation*}
as $b\rightarrow 0$ and $N\rightarrow \infty$. Furthermore, equation
(\ref{h function in use 2}) gives
\begin{equation}\label{h0 pm i and 0}
    \begin{split}
    h_0(\pm
    i,0)=\frac{-(1-a)v_\mp}{b}\sqrt{\frac{2(1-2a)m}{b\sqrt{4a-4a^2-b^2}(1-v_\mp)}}e^{\pm \Delta
    i},
    \end{split}
\end{equation}
where $\Delta=2\arctan\sqrt{\frac{a}{1-a}}-\sqrt{\frac{a}{1-a}}$.

\subsection{The mapping $v\rightarrow u$ in (\ref{mapping in case 3})}\label{sec 4}
In the case under discussion, $a/b^2\rightarrow \infty$ and
$x\rightarrow\infty$. Thus, since $0<a<\frac{1}{2}$, we have
$4a-4a^2-b^2>0$, and the saddle points in (\ref{saddle points of v})
are complex.

\begin{thm}\label{thm 2}
When $a\in(0,\frac{1}{2}]$, $b\rightarrow 0^+$ and $b^2/a\rightarrow
0$, the mapping $v\rightarrow u$ defined in (\ref{mapping in case
3}) is one-to-one and analytic for $v\in D_v$ in the $v$-plane and
$u\in D_u$ in the $u$-plane, where the image of the boundary of
$D_u$ in the $Z$-plane is given in (\ref{boundaries in the
figures}), and $D_v$ is the image of $D_u$ under the mapping
(\ref{mapping in case 3}).
\end{thm}
\begin{proof}
As in the cases of Charlier polynomials \cite{Rui} and Meixner
polynomials \cite{xiaojin}, we introduce an intermediate variable
$Z$ defined by
\begin{equation}\label{big Z in proof in case 3}
\begin{split}
   b\ln(1-t_0(v))+&(1-b)\ln t_0(v)-a\ln (1-v)-b\ln(\frac{v}{t_0(v)+v-1})-\gamma\\
    =&Z=m\left(u-\frac{1}{u}\right),
    \end{split}
\end{equation}
where $t_0(v)$ is the relevant saddle point given in (\ref{saddle
point of t in case 3}). \\

\begin{figure}[!h]
\centering
  \includegraphics[scale=0.7]{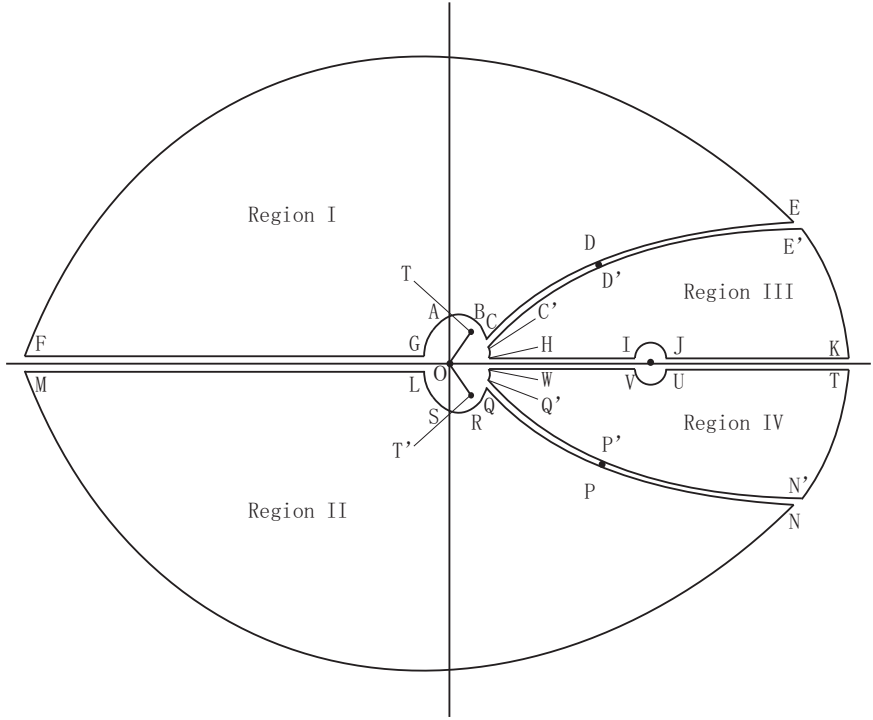}\\
  \caption{The upper half of the $v$-plane.}\label{w-proof}
\end{figure}

\begin{figure}[!h]
\centering
  \includegraphics[scale=0.6]{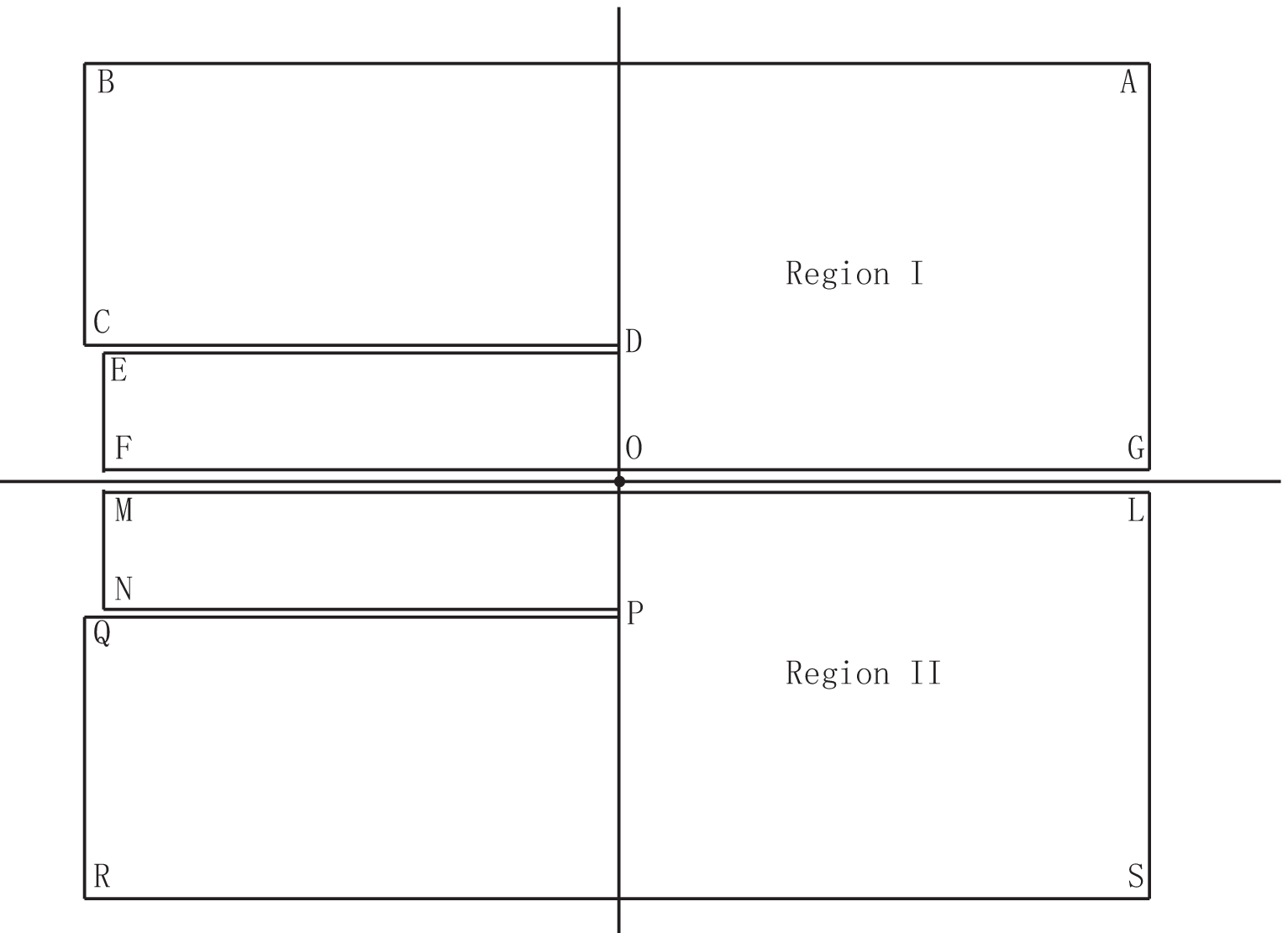}\\
  \caption{The image of Region I in the $Z$-plane.}\label{z-proof-I}
\end{figure}

\begin{figure}[!h]
\centering
  \includegraphics[scale=0.6,bb=0 0 420 340]{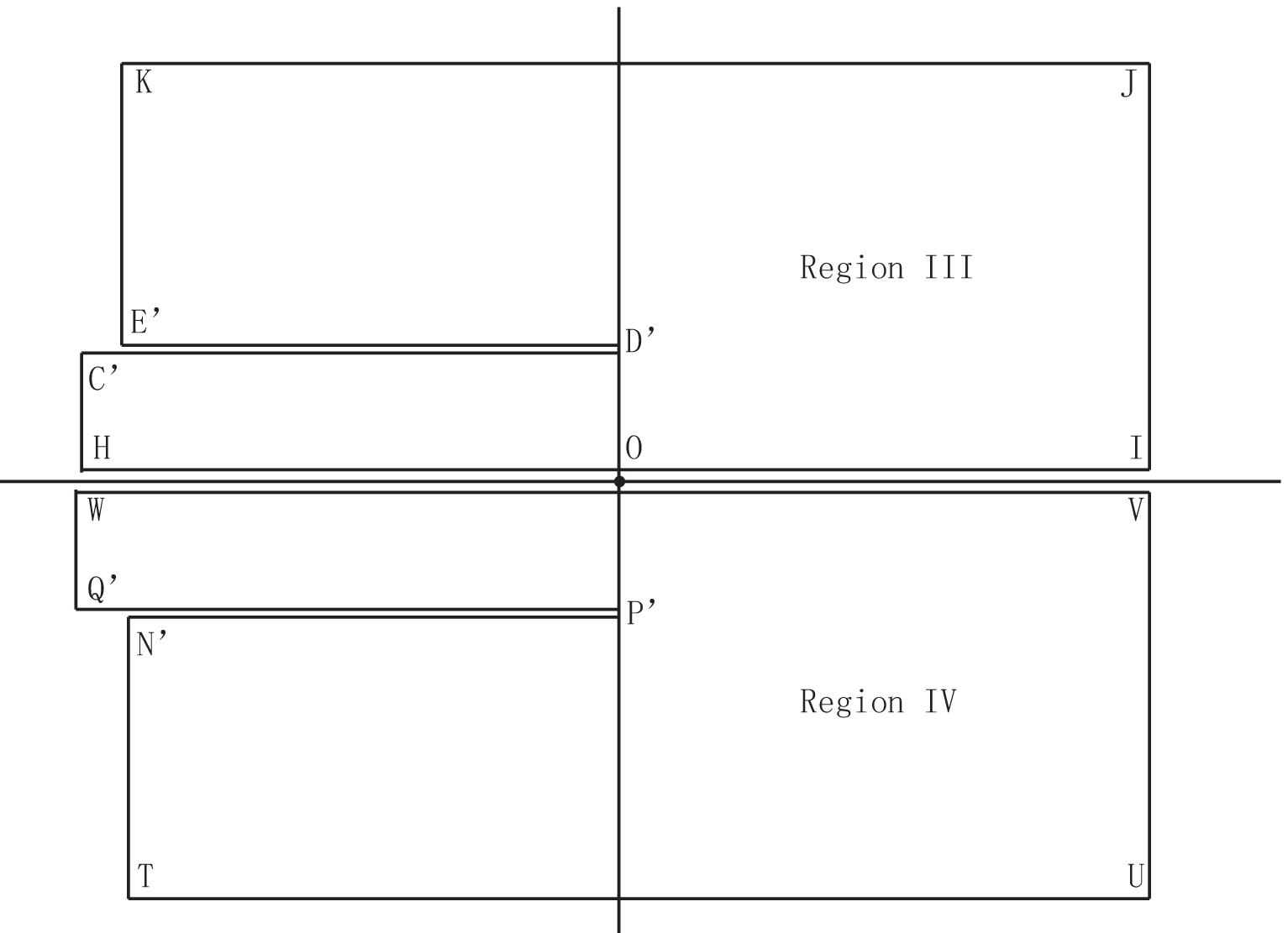}\\
  \caption{The image of Region II in the $Z$-plane.}\label{z-proof-II}
\end{figure}

\begin{figure}[!h]
\centering
  \includegraphics[scale=0.6]{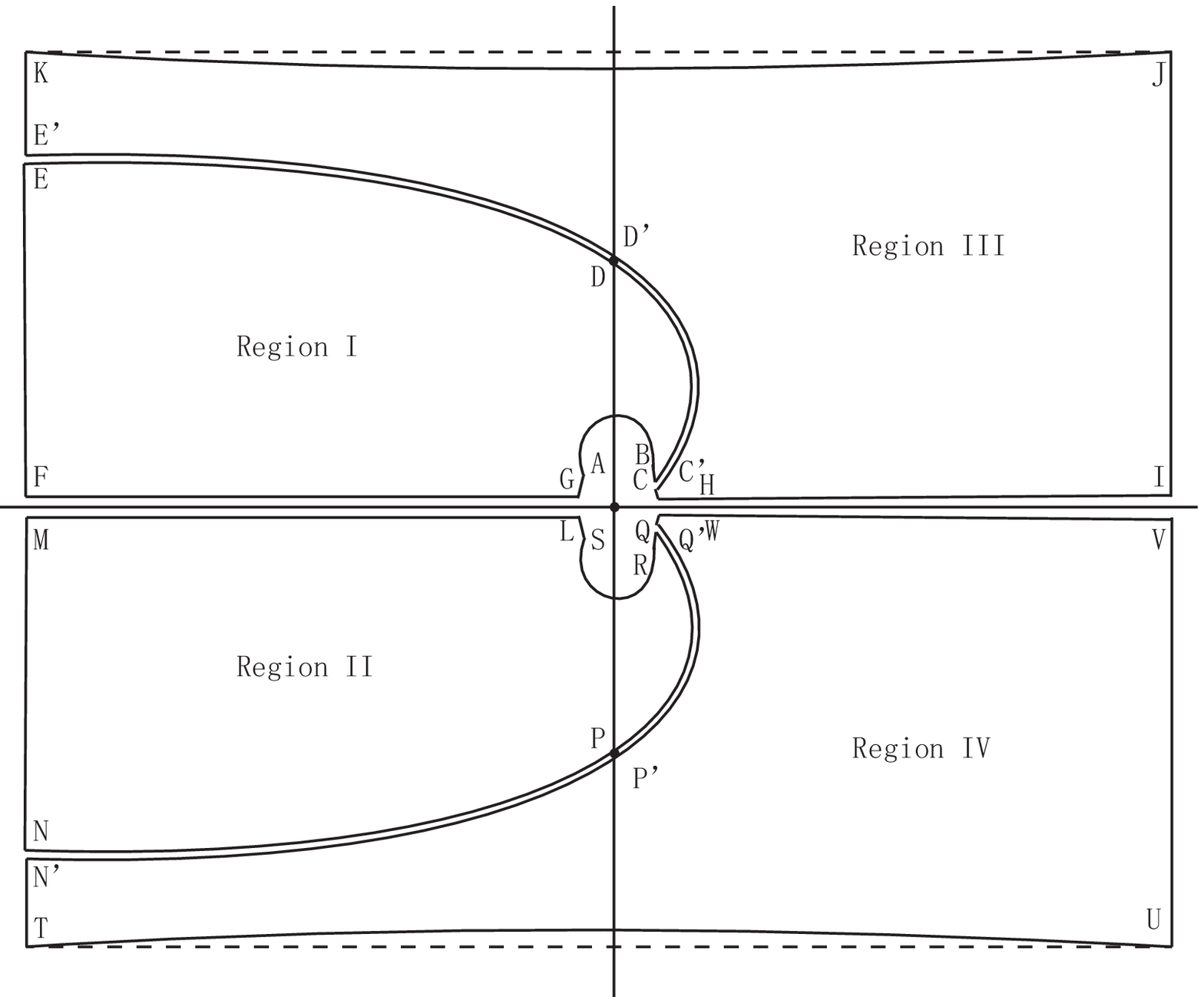}\\
  \caption{The upper half of the $u$-plane.}\label{u-proof}
\end{figure}

We consider the upper half of the $v$-plane. Since the functions
$\widehat{f}(t_0(v),v)-\gamma$ and $g(u)$ are both symmetric with
respect to the real line, the case for the lower half of the
$v$-plane can be handled in the same manner. To avoid
multi-valuedness around the saddle point, we divide the upper half
of the $v$-plane into two parts by using the steepest descent path
through $v_-$, i.e., the point denoted by $D$; see Figure
\ref{w-proof}. Call the two parts region I and region III. Note that
there are two branch cuts: one along the infinite interval
$[1,\infty)$, and the other along the line segment joining $v=0$ to
$v=2b^2+i 2b\sqrt{1-b^2}$, which was introduced in (\ref{saddle
point of t in case 3}). We denoted the point $2b^2+i 2b\sqrt{1-b^2}$
by the letter $T$ in Figure 4.\\

As $v$ traverses along the boundary $ABCDEFGA$ of region I, the
image point $Z$ traverses along the corresponding boundary of a
region in the $Z$-plane; see Figure \ref{z-proof-I}. In Figure
\ref{u-proof}, we draw the boundary of the region, corresponding to
region I, in the $u$-plane. The image point $Z=\psi(u)+\gamma$
traverses along the boundary of the same region shown in Figure
\ref{z-proof-I}. Note that we have used an arc $GABC$ to avoid the
cut $OT$. The boundary curves $EF$ and $GABC$ in the $v$-plane are
rather arbitrary; for convenience, we choose them to be the ones
whose images
in the $Z$-plane are as indicated in Figure 5.\\

From (\ref{mapping in case 3}), it is readily seen that the mapping
$v\rightarrow u$ is the composite function of $g^{-1}$ and
$\widehat{f}(t_0(v),v)-\gamma$. We have just verified that this
mapping is one-to-one on the boundary of region I. By the same
argument, one can prove that this mapping is also one-to-one on the
boundary of region III. For the image of region III in the
intermediate $Z$-plane, see Figure \ref{z-proof-II}. In the
$u$-plane, we let $D_u$ denote the union of the regions I to IV,
with outer boundary $FEE'KJIVUTN'NM$, and inner boundary
$GABCC'HWQ'QRSL$; see Figure \ref{u-proof}. Under the mapping
$Z=g(u)$, we express the images of parts of the outer and inner
boundaries of $D_u$ in the $Z$-plane as follows:
\begin{equation}\label{boundaries in the figures}
\begin{split}
  GA:\qquad & \text{Re }Z=bM  \hspace{1.5cm} \qquad 0\leq \text{Im }Z\leq \theta b \pi; \\
  AB:\qquad  & -bM\leq \text{Re }Z\leq bM \hspace{1.2cm} \text{Im }Z= \theta b\pi; \\
  BC\text{ and } C'H:\qquad  & \text{Re }Z=-bM \hspace{1.9cm} 0\leq \text{Im }Z\leq \theta b\pi; \\
  IJ:\qquad  & \text{Re }Z=aM \hspace{2.22cm} 0\leq \text{Im }Z\leq a\pi; \\
  JK:\qquad  & -aM\leq \text{Re }Z\leq aM \hspace{1.2cm}  \text{Im }Z= a\pi; \\
  KF:\qquad & \text{Re }Z=-aM  \hspace{2cm} 0\leq \text{Im }Z\leq
  a\pi,
  \end{split}
\end{equation}
where $M>0$ is any fixed constant, and $\theta$ is a generic symbol
for a constant in
$(0,1)$. Furthermore, we denote by $D_v$ the corresponding region of $D_u$ in the $v$-plane.\\

By Theorem 1.2.2 of \cite[p.12]{lecture notes}, the mapping
$v\rightarrow u$ is one-to-one in the interior of both regions I and
III in the upper half of the $v$-plane. As explained earlier, the
one-to-one property of this mapping in the lower half of the
$v$-plane can be established by using the symmetry of the functions
with respect to the real axis. Note that the only possible singular
points of the mapping $v\rightarrow u$ in $D_v$ are at $v=v_\mp$.
Since the images of these points in the $u$-plane (i.e., $u=\pm i$)
are bounded, the mapping is in fact one-to-one and analytic in
$D_v$.

\end{proof}

\subsection{Analyticity of $h_0(u,\tau)$}
We now investigate the function $h_0(u,\tau)=h(u,\tau)$ given in
(\ref{h function after the mapping in case 3}). By (\ref{dt over
dtau b-0}) and (\ref{partial v over partial u in case 3}), we have
\begin{equation}\label{h function in use}
h_0(u,\tau)=\frac{m(\tau-b)
t(1-t)(t+v-1)(u-i)(u+i)\left[(1-2a)v+\sqrt{4b^2-4b^2
v+v^2}\right]}{2a(1-a)\tau(1+b)(t-t^+_0(v))(t-t^-_0(v))u(v-v_+)(v-v_-)}
\end{equation}
when $\tau\neq b$ and $u\neq \pm i$, and
\begin{equation}
\begin{split}
\label{ }
h_0(u,\tau)=&\frac{m(u-i)(u+i)\left[(1-2a)v+\sqrt{4b^2-4b^2
v+v^2}\right]}{2a(1-a)u(v-v_+)(v-v_-)}\\
&\times\sqrt{\frac{t_0(v)(1-t_0(v))(t_0(v)+v-1)}{b\sqrt{4b^2-4b^2
v+v^2}}}
\end{split}
\end{equation}
when $\tau= b$ and $u\neq \pm i$, where $u\in D_u$ and $D_u$ is
defined in the previous subsection. Note that $u=0$ is excluded from
$D_u$. It is easy to see that $(t-1)/\tau$ is bounded by a constant,
independent of $b$ and $a$, as $\tau\rightarrow 0$ for $v\in D_v$
and $u\in D_u$. Therefore, $h_0(u,\tau)$ is analytic in $\tau$ for
$\tau$ in the neighbourhood of $[0,\infty$). Moreover, $h_0(u,\tau)$
is also analytic in $u$ for $u\in D_u$, since the only singularities
$u=\pm i$ in $D_u$ are both removable; indeed, we have
\begin{equation}
\label{h function in use 2} h_0(u,\tau)=\frac{(\tau-b)
t(1-t)(t+v_\mp-1)}{\tau(1+b)(t-t^+_0(v_\mp))(t-t^-_0(v_\mp))}\sqrt{\frac{2(1-2a)m}{(1-v_\mp)b\sqrt{4a-4a^2-b^2}}}
\end{equation}
when $\tau\neq b$ and $u= \pm i$, and by the equation following
(\ref{partial v over partial u in case 3}) we also have
\begin{equation}
\label{h function in saddlepoint}
h_0(u,\tau)=\sqrt{\frac{2mt_\mp(1-t_\mp)(t_\mp+v_\mp-1)}{b^2v_\mp(1-v_\mp)\sqrt{4a-4a^2-b^2}}}
\end{equation}
when $\tau= b$ and $u= \pm i$. To reach (\ref{h function in
saddlepoint}), we have written $u_\pm$ as $u_\pm=\sqrt{u_\pm^2}$ and
move $u_\pm^2$ inside the square root in $\mathrm{d}v/\mathrm{d}u$
at $u=u_\pm$; see
(\ref{partial v over partial u in case 3}).\\

For the analysis to be used in the next subsection, we now give an
estimate for $h_0(u,0)$ when $u$ lies on the boundaries of $D_u$,
i.e, the outer boundary $FEE'KJIVUTN'NM$ and the inner boundary
$GABCC'HWQ'QRSL$ shown in Figure \ref{u-proof}. For convenience, let
us denote the outer boundary by $C_O$ and the inner boundary by
$C_I$.\\

First, we show that for $u$ on the inner boundary $C_I$,
\begin{equation}\label{u estimate on CI}
    C_1 \frac{m}{b}<|u|<C_2\frac{m}{b},
\end{equation}
where $C_1<C_2$ and $C$, $C_1$, $C_2$ are used, here and thereafter,
as generic symbols for constants independent of $u$, $v$, $m$, $a$
and $b$. Recall from (\ref{asymptotic behaviour of m}) that $m\sim
b\arctan \sqrt{\frac{a}{1-a}}$ as $b\rightarrow 0$. To prove (\ref{u
estimate on CI}), we take the part $AB$ of the inner boundary $C_I$
as an illustration. Using (\ref{boundaries in the figures}), we
obtain
\begin{equation}
\label{temp for estimate of U}
\left|m\left(1-\frac{1}{|u|^2}\right)\text{Re }u\right|\leq
bM\quad\text{and}\quad m\left(1+\frac{1}{|u|^2}\right)\text{Im
}u=\theta b \pi.
\end{equation}
Since $\left|\text{Im }u\right|\leq \left|u\right|$, from the
equality in (\ref{temp for estimate of U}) we have
$|u|^2-\frac{\theta b\pi}{m}|u|+1\geq 0$. Therefore,
\begin{equation}
\label{temp for estimate of U 2} |u|\leq \frac{\theta b\pi}{2
m}-\sqrt{\left(\frac{\theta b\pi}{2m}\right)^2-1}\leq
\frac{Cm}{b\pi}.
\end{equation}
Rewriting (\ref{temp for estimate of U}) gives
\begin{equation}
\label{ } \frac{m}{|u|^2}\left|\text{Re }u\right|\leq
bM+m\left|\text{Re }u\right|\quad\text{and}\quad
\frac{m}{|u|^2}\left|\text{Im }u\right|\leq \theta
b\pi+m\left|\text{Im }u\right|,
\end{equation}
which lead to $\frac{m}{|u|}\leq b M+\theta b\pi+m|u|$. Substituting
(\ref{temp for estimate of U 2}) into the right-hand side of the
last inequality, and combining the resulting inequality with
(\ref{temp for estimate of U 2}), we obtain (\ref{u estimate on CI})
immediately. Similarly, for $u$ on the outer boundary $C_O$, we have
the estimates
\begin{equation}\label{u estimate on CO}
    C_1 \frac{a}{m}<|u|<C_2\frac{a}{m}.
\end{equation}

Hence, it follows from (\ref{u estimate on CI}), (\ref{u estimate on
CO}) and (\ref{h function in use}) that
\begin{equation}\label{estimate of h on CI and CO}
    |h(u,0)|\leq C\qquad\text{ for }u\in C_I\cup C_O.
\end{equation}

\subsection{Error bounds for the remainder}
To prove the asymptotic nature of the expansion in (\ref{result a>0
in case 3}), we need to give precise estimates for its coefficients
$c_l$, $d_l$ in (\ref{cl when a>0 in case 3}) and (\ref{d_l when a>0
in case 3}), and the error term $\varepsilon_p^+$ given in
(\ref{epsilon p in case 3}), since the derivative
$\mathrm{d}v/\mathrm{d}u$ in (\ref{partial v over partial u in case
3}) may blow up as $v_\pm$ approach 0, just like the case in Section
4. Because the coefficients $c_l$ and $d_l$ are related to
$a_p(\tau)$ and $b_p(\tau)$ by (\ref{expansion of coeffiencet an and
bn in case 3}), (\ref{cl when a>0 in case 3}) and (\ref{d_l when a>0
in case 3}), let us first estimate $a_p(\tau)$ and $b_p(\tau)$. To
this end, we define recursively
\begin{equation*}
    A_0(u)=\frac{u}{1+u^2},\qquad B_0(u)=-\frac{1}{1+u^2},
\end{equation*}
and
\begin{eqnarray*}
  A_{p+1}(u) &=& \frac{1}{m}\left(1+\frac{1}{u^2}\right)^{-1}\left\{\frac{1}{u}+\frac{\mathrm{d}}{\mathrm{d}u}\right\}A_p(u),\\
  B_{p+1}(u) &=& \frac{1}{m}\left(1+\frac{1}{u^2}\right)^{-1}\left\{\frac{1}{u}+\frac{\mathrm{d}}{\mathrm{d}u}\right\}B_p(u)
\end{eqnarray*}
for $p=0,1,2,...$. By induction, it can be shown that
\begin{equation}\label{Ak and Bk}
\begin{split}
  A_{p}(u) &= \left(\frac{1}{2m}\right)^p\left(\frac{1}{u}\right)^{p+1}\left(\frac{u^2}{1+u^2}\right)^{p+1}\sum_{l=0}^p c_{p,l}\left(\frac{u^2}{1+u^2}\right)^{l},\\
  B_{p}(u) &= \left(\frac{1}{2m}\right)^p\left(\frac{1}{u}\right)^{p+2}\left(\frac{u^2}{1+u^2}\right)^{p+1}\sum_{l=0}^p d_{p,l}\left(\frac{u^2}{1+u^2}\right)^{l}
  \end{split}
\end{equation}
for $k=0,1,2,...$, where $c_{k,l}$ and $d_{k,l}$ are constants
independent of $m$ and $u$ (but dependent of $\tau$); see
\cite{zhao}. Furthermore, by (\ref{al bl and hl}), (\ref{h_l in case
3}) and (\ref{hl 2 in case 3}), it can be proved that
\begin{equation}\label{al integral}
    a_p(\tau)=\frac{1}{2\pi i}\int_{C_O\bigcup C_I}h_0(u,\tau)A_p(u)\mathrm{d}u
\end{equation}
and
\begin{equation}\label{bl integral}
    b_p(\tau)=\frac{1}{2\pi i}\int_{C_O\bigcup
    C_I}h_0(u,\tau)B_p(u)\mathrm{d}u;
\end{equation}
see also \cite{zhao}. Define $q:=\min\{a,b\}$; since in this case
both $n$ and $x\rightarrow \infty$, and $b^2/a\rightarrow 0$, we
have $qN\rightarrow\infty$, $m/q=O(b/\sqrt{a})\rightarrow 0$ if
$a\leq b$, and $m/q\sim \arctan
\sqrt{a/(1-a)}$ if $a> b$.\\

Using (\ref{Ak and Bk}), (\ref{u estimate on CI}) and (\ref{u
estimate on CO}), it is easily verified that
\begin{equation}\label{inner boundary Ak and Bk}
    \left|A_p(u)\right|\leq \frac{C m
    }{b^{p+1}},\qquad\left|B_p(u)\right|\leq \frac{C
    }{b^{p}}
\end{equation}
for $u$ on the inner boundary $C_I$, and
\begin{equation}\label{outer boundary Ak and Bk}
    \left|A_p(u)\right|\leq \frac{C m
    }{a^{p+1}},\qquad\left|B_p(u)\right|\leq \frac{C m^2
    }{a^{p+2}}
\end{equation}
for $u$ on the outer boundary $C_O$, where $C$ is used again as a
generic symbol for constants independent of $u$, $a$ and $b$. Also
we have the estimates
\begin{equation}\label{length of the boundaray}
\int_{C_I}\left|\mathrm{d}s\right|\leq \frac{C
m}{b}\quad\text{and}\quad \int_{C_O}\left|\mathrm{d}s\right|\leq
\frac{C a}{m}.
\end{equation}
To see this, we take one part of $C_I$, namely, $AB$ as an example;
see (\ref{big Z in proof in case 3}) and (\ref{boundaries in the
figures}), where $Z$ in  (\ref{big Z in proof in case 3}) satisfies
$\text{Im } Z=\theta b\pi$. Using the second equality in
(\ref{boundaries in the figures}), and in polar coordinates, $u=r
e^{i\beta}$, the last equation is equivalent to
\begin{equation}
\label{ } \frac{(r^2+1)}{r}\sin \beta=\frac{\theta b\pi }{m}.
\end{equation}
Thus
\begin{equation}
\label{ } \mathrm{d} r=\frac{(1+r^2)^2 m}{(1-r^2)\theta b\pi}\cos
\beta \mathrm{d}\beta.
\end{equation}
Since in this case $r< 1$, we have $\left|\mathrm{d}r\right|\leq
\frac{Cm}{ b}\left|\mathrm{d}\beta\right|$. By (\ref{u estimate on
CI}), $|r|\leq C_2m/b$ and
\begin{equation}
\label{ }
\left|\mathrm{d}s\right|=\sqrt{(\mathrm{d}r)^2+r^2(\mathrm{d}\beta)^2}\leq
\frac{C m}{b}\left|\mathrm{d}\beta\right|,
\end{equation}
from which the first inequality in (\ref{length of the boundaray})
follows. The second inequality can be established in a similar
manner.\\

Since $h_0(u,0)\neq 0$ and $h_0(u,\tau)$ is analytic in $\tau$, we
have $|h_0(u,\tau)|\leq C |h_0(u,0)|$ for some constant $C>0$ and
for $\tau$ in the neighbourhood of the origin. By a combination of
(\ref{al integral}), (\ref{bl integral}), (\ref{inner boundary Ak
and Bk}), (\ref{outer boundary Ak and Bk}), (\ref{length of the
boundaray}) and (\ref{estimate of h on CI and CO}), we obtain
\begin{equation}\label{al and bl estimate}
    |a_p(\tau)|\leq \frac{C}{q^{p}}\quad
    \text{and}\quad |b_p(\tau)|\leq \frac{C}{q^{p}}
\end{equation}
for $\tau$ near $\tau=0$. Furthermore, by (\ref{expansion of
coeffiencet an and bn in case 3}), (\ref{cl when a>0 in case 3}),
(\ref{d_l when a>0 in case 3}) and (\ref{al and bl estimate}), we
have
\begin{equation}\label{cl dl estimate in bessel}
\left|c_p\right|\leq
    \frac{C}{q^{p}}\quad
    \text{and}\quad \left|d_p\right|\leq
    \frac{C}{q^{p}}.
\end{equation}

To estimate the remainder in (\ref{epsilon p in case 3}), we split
the loop contour into two parts; see Figure \ref{D_v-in-proof}. The
bounded part of the contour, denoted by $\Gamma_1$, is contained in
a subdomain $\widetilde{D}_u$ of $D_u$, which has a distance $c_1
a/m$ from the outer boundary of $D_u$ (i.e., $C_O$), and has a
distance $c_2 m /b$ from the inner boundary of $D_u$ (i.e., $C_I$),
$c_1$ and $c_2$ being two constants independent of $u$, $a$ and $b$.
Note that $c_1a/m$ may become large, whereas $c_2m/b$ may approach
zero. The unbounded part of the contour, denoted by $\Gamma_2$, is
the rest of the loop outside the subdomain $\widetilde{D}_u$. Put
\begin{eqnarray}
\varepsilon_{p,1}^+&= &\frac{N^{n+1-p}}{2\Gamma(n+1)\pi
i}\int_{0}^{+\infty}\int_{\Gamma_1}\frac{h_p(u,\tau)}{u}e^{N
m\left(u-\frac{1}{u}\right)}\tau ^n
e^{-N\tau}\mathrm{d}u\mathrm{d}\tau,\label{error term in bessel function in proof} \\
\varepsilon_{p,2}^+&= &\frac{N^{n+1-p}}{2\Gamma(n+1)\pi
i}\int_{0}^{+\infty}\int_{\Gamma_2}\frac{h_p(u,\tau)}{u}e^{N
m\left(u-\frac{1}{u}\right)}\tau ^n
e^{-N\tau}\mathrm{d}u\mathrm{d}\tau.
\end{eqnarray}
\begin{figure}[!h]
\centering
  \includegraphics[scale=0.55,bb=0 0 488 410]{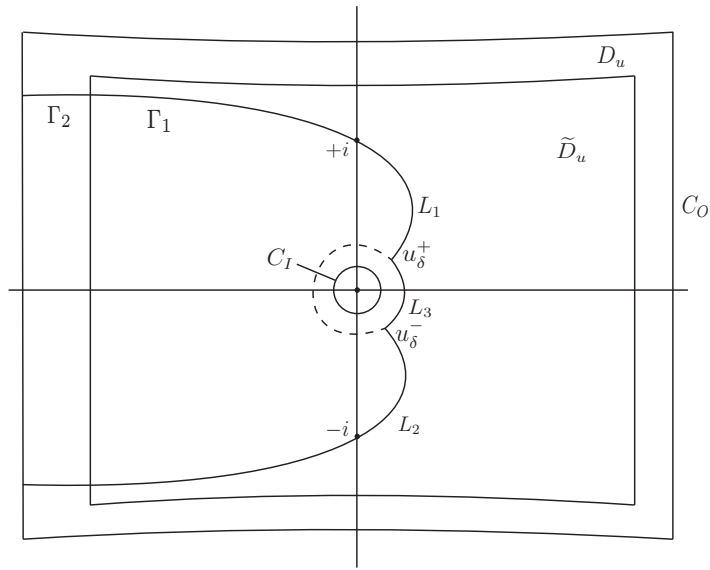}\\
  \caption{The contour in $u$-plane.}\label{D_v-in-proof}
\end{figure}\\
As in \cite{Olde Daalhuis} and \cite{zhao}, it can be shown that
$\varepsilon_{p,2}^+$ is exponentially small in comparison with
$\varepsilon_{p,1}^+$. To estimate $\varepsilon_{p,1}^+$, we also
follow \cite{Olde Daalhuis} and \cite{zhao} by first deriving an
integral representation for $h_k(u,\tau)$. Define recursively
\begin{equation}\label{Q in integral representation}
    \begin{split}
        Q_0(u,w)&=\frac{1}{u-w},\\
        Q_{p+1}(u,w)&=\frac{1}{m}\left(1+\frac{1}{u^2}\right)^{-1}\left\{\frac{1}{u}+\frac{\mathrm{d}}{\mathrm{d}u}\right\}Q_{p}(u,w)
    \end{split}
\end{equation}
for $p=0,1,2,\cdots$. Using induction, one can easily verify that
\begin{equation}\label{Qk expression}
    Q_{p}(u,w)=\frac{1}{(2m)^p}\sum_{i=0}^{p-1}\sum_{j=0}^{p-i}q_{p,i,j}\left(\frac{1}{1+u^2}\right)^{p+i}u^{2p+i-j}\left(\frac{1}{u-w}\right)^{p+1-i-j},
\end{equation}
where $q_{p,i,j}$ are some constants depending only on $i$, $j$ and
$p$. Similar to (\ref{IR of h-k}), for $w\in \widetilde{D}_u$ we
have from (\ref{h_l in case 3}), (\ref{hl 2 in case 3}) and (\ref{Q
in integral representation}) the integral representation
\begin{equation}\label{hk(w)}
    h_p(w,\tau)=\frac{1}{2\pi i}\int_{C_I\cup
    C_O}h_0(u,\tau)Q_p(u,w)\mathrm{d}u.
\end{equation}
From (\ref{Qk expression}), it is easy to see that
\begin{equation}\label{estimate of Q}
\begin{split}
  |Q_p(u, w)| &\leq  C \frac{m}{a^{p+1}}\quad \text{ for }u\in C_O\text{ and }w\in \widetilde{D}_u, \\
  |Q_p(u, w)| &\leq C \frac{1}{m b^{p-1}}\quad \text{ for }u\in C_I\text{ and }w\in
  \widetilde{D}_u.
\end{split}
\end{equation}
Coupling the above inequalities with (\ref{length of the boundaray})
and (\ref{hk(w)}), we obtain
\begin{equation}\label{h_k estimate}
    |h_p(w,\tau)|\leq \frac{C}{q^{p}}
\end{equation}
for $w\in \tilde{D}_u$ and $\tau$ in the neighbourhood of $\tau=0$.\\

If $Nm$ is bounded, one can easily show from (\ref{error term in
bessel function in proof}) and (\ref{h_k estimate}) that
\begin{equation}\label{I_for Nm bounded}
    |\varepsilon_{p,1}^+|=O\left(\frac{1}{q^{p}N^p}\right).
\end{equation}

If $Nm\rightarrow \infty$ we divide the integration path $\Gamma_1$
into three pieces: the steepest descent path $L_1$ through $u=u_+=+
i$ in the upper half of the $u$-plane, the steepest descent path
$L_2$ through $u=u_-=-i$ in the lower half of the $u$-plane, and the
circular arc $|u|=1-\delta_0$, denoted by $L_3$, which joins the
steepest descent paths $L_1$ and $L_2$ at $u=u_\delta^+$ and
$u=u_\delta^-$, respectively; see Figure \ref{D_v-in-proof}. Hence,
\begin{equation}\label{}
\varepsilon_{p,1}^+=I_1+I_2+I_3,
\end{equation}
where $I_i$, $i=1,2,3$, denotes the integral over the subcontour
$L_i$. Applying the steepest descent method \cite[p.84]{R.Wong's
book}, and using (\ref{h_k estimate}), it can be easily verified
that
\begin{equation}\label{I_1,2}
    |I_{1,2}|=O\left(\frac{1}{q^{p}N^p}\sqrt{\frac{1}{Nm}}\right).
\end{equation}
For $I_3$, since
\begin{equation*}
\begin{split}
\text{Re}\left(u-\frac{1}{u}\right)&=-\frac{(2-\delta)\delta}{(1-\delta)^2}\text{Re
}u\leq-\frac{(2-\delta)\delta}{(1-\delta)^2}\text{Re
}u_\delta^\pm\\
&=\left.\text{Re}\left(u-\frac{1}{u}\right)\right|_{u=u_\delta^\pm}<\text{Re}\left.\left(u-\frac{1}{u}\right)\right|_{u=u_\pm},
\end{split}
\end{equation*}
$I_2$ is exponentially small in comparison with $I_{1,2}$.
Therefore,
\begin{equation}\label{epsilon p 1 estimate}
   |\varepsilon_{p,1}^+|=O\left(\frac{1}{q^{p}N^p}\sqrt{\frac{1}{Nm}}\right).
\end{equation}
It is well-known that the Bessel function $ J_\alpha(t)$ is bounded
when $t$ is bounded, and that as $t\rightarrow\infty$,
\begin{equation}\label{bessel asymptotic}
    J_\alpha(t)\sim\sqrt{\frac{2}{\pi
    t}}\cos\left(t-\frac{1}{2}\alpha\pi-\frac{1}{4}\pi\right).
\end{equation}
Coupling the estimates (\ref{I_for Nm bounded}) and (\ref{epsilon p
1 estimate}), and together with (\ref{bessel asymptotic}) and
(\ref{cl dl estimate in bessel}), we can find a constant $M$
independent of $a$ and $b$ such that
\begin{equation}\label{remainder estimate}
|\varepsilon_{p}^+|\leq
\frac{M}{qN}\left(|J_0(2Nm)|\frac{|c_{p-1}|}{N^{p-1}}+|J_1(2Nm)|\frac{|d_{p-1}|}{N^{p-1}}\right),
\end{equation}
which establishes the asymptotic nature of the expansion in
(\ref{result a>0 in case 3}), given that $qN\rightarrow \infty$. In
conclusion, in the case $n^2/xN\rightarrow 0$ and $x\rightarrow
\infty$, (\ref{result a>0 in case 3}) is an asymptotic expansion.

\section{REMAINING CASES}
We are now left with only two easy cases to consider.
\subsection{$a>0$, $a/b^2$ large and $x$ bounded}
In this case, the series in (\ref{relation between DCP and Hahn}) is
itself an asymptotic expansion, since
\begin{equation*}
    \varphi_k(N):=\frac{(-n)_k(-x)_k(n+1)_k}{(-N+1)_k
    k!k!},\qquad\qquad k=0,1,\cdots,n,
\end{equation*}
is an asymptotic sequence when $xN/n^2\rightarrow \infty$; see Table
1 and also \cite[p.10]{R.Wong's book}.

\subsection{$a<0$} The case $a<0$ is relatively easy, compared with
the case $a>0$. Let us start with the integral representation
(\ref{IR of DCP for x<0}). For fixed $w\in \gamma_2$, the phase
function $\widetilde{f}(t,w)$ in (\ref{IR of DCP for x<0}) has a
saddle point
\begin{equation}\label{saddle points for t a<0 -}
    t_0(w)=\frac{2w-1+\sqrt{1+4b^2(w-1)w}}{2(1+b)w};
\end{equation}
cf.(\ref{saddle points for t}). Note that the only difference
between $f(t,w)$ in (\ref{phase function f for x>0}) and
$\widetilde{f}(t,w)$ in (\ref{phase function f for x<0}) is the
choice of branch cuts in the $w$-plane due to the term $a\ln
w-a\ln(w-1)$ in $f(t,w)$ and the term $a\ln(-w)-a\ln(1-w)$ in
$\widetilde{f}(t,w)$. Similar to (\ref{property of t0(w)}), for any
fixed $w\in \gamma_2$, we have $t_0(w)=1-b+O(b^2)$ as $b\rightarrow
0$; in particular, $t_0(0)=1-b$. By the same reasoning given for
(\ref{mapping t to tau case 1}), we introduce the mapping
$t\rightarrow\tau$ defined by
\begin{equation}\label{mapping t to tau a<0 -}
    b \ln \tau-\tau+A=\widetilde{f}(t,w)-\widetilde{f}(t_0(w),w)
\end{equation}
with
\begin{equation}\label{temp123}
    \tau(t_0(w))=b.
\end{equation}
Coupling (\ref{temp123}) and (\ref{mapping t to tau a<0 -}) yields
$A=b-b\ln b$. From (\ref{mapping t to tau a<0 -}), we have
\begin{equation}\label{dt over dtau a<0-}
        \begin{split}
            \frac{\mathrm{d}t}{\mathrm{d}\tau}& =\frac{b/\tau-1}{\partial \widetilde{f}(t,w)/\partial
            t}=\frac{(\tau-b) t(1-t)[1-(1-t)w]}{\tau(1+b)w(t-t^+_0(w))(t-t^-_0(w))},\quad\quad \tau\neq
            b,\\
            \frac{\mathrm{d}t}{\mathrm{d}\tau}&=\left\{\frac{t^+_0(w)(1-t^+_0(w))[1-(1-t^+_0(w))w]}{b\sqrt{1+4b^2(w-1)w}}\right\}^{1/2}, \hspace{1.6cm}
            \tau=b,
        \end{split}
\end{equation}
where we have used L'H$\hat{o}$spital's rule for $\tau=b$ and
\begin{equation}\label{}
    t_0^{\pm}(w)=\frac{2w-1\pm\sqrt{1+4b^2(w-1)w}}{2(1+b)w};
\end{equation}\\
cf.(\ref{dt over dtau b-0 case 1}). Furthermore, it follows from
(\ref{mapping t to tau a<0 -}) and (\ref{IR of DCP for x<0}) that
\begin{equation}\label{IP after the t mapping 2}
\begin{split}
t_n(x,N+1)=& \frac{(-1)^n}{2\pi
i}\frac{\Gamma(n+N+2)e^nb^{-n}}{\Gamma(n+1)\Gamma(N-n+1)}\\
&\times \int_{0}^{+\infty}\int_{\gamma_2}\frac{1}{w-1}e^{N
\widetilde{f}({t_0(w)},w)}e^{N(b\ln
\tau-\tau)}\frac{\mathrm{d}t}{\mathrm{d}
\tau}\mathrm{d}w\mathrm{d}\tau;
\end{split}
\end{equation}
cf.(\ref{IP after the t mapping b-0}). Solving the equation
$\partial f(t_0(w),w)/\partial w=0$, we obtain the saddle points
\begin{equation}\label{saddle points of w a<0 2}
\begin{split}
    w_{\pm}=&\frac{b\pm\sqrt{b^2-4a+4a^2}}{2b}\\
    =&\frac{1}{2}\pm\frac{1}{2}\sqrt{1-\frac{4a(1-a)}{b^2}};
\end{split}
\end{equation}
cf.(\ref{saddle points of w}). The relevant saddle point on the
integral path $\gamma_2$ is the negative saddle point $w_-$.\\

The following procedure is the same as that given in \cite{pan}.
Recall the Hankel integral for the Gamma function
\begin{equation}\label{}
    \frac{1}{\Gamma(z)}=\frac{1}{2\pi
    i}\int^{(0+)}_{-\infty}e^uu^{-z}\mathrm{d}u,
\end{equation}
where the contour is a loop starting at $-\infty$, encircling the
origin in the counterclockwise direction and returning to $-\infty$.
With $z=-aN+1$ and $u$ replaced by $-u$, we obtain
\begin{equation}\label{Gamma function}
    \frac{1}{\Gamma(-aN+1)}=\frac{N^{aN}}{2\pi
    i}\int^{(0+)}_{\infty}\frac{1}{u}e^{N\widetilde{\psi}(u)}\mathrm{d}u,
\end{equation}
where $\widetilde{\psi}(u)=a\ln(-u)-u$. Make the transformation
$w\rightarrow u$ defined by
\begin{equation}\label{mapping w to u when a<0}
    \begin{split}
    \widetilde{f}(t_0(w),w)&=a\ln(-u)-u+\gamma
    \end{split}
\end{equation}
with
\begin{equation}\label{relation betwen w and u when a<0}
    u(w_-)=a,
\end{equation}
where $\gamma$ is a constant to be determined. We have from
(\ref{mapping w to u when a<0})
\begin{equation}\label{}
        \begin{split}
            \frac{\mathrm{d}w}{\mathrm{d}u}&=\frac{(a-u)w(w-1)\left[(1-2a)+\sqrt{4b^2 w^2-4b^2 w+1}\right]}{2ub^2(w-w_+)(w-w_-)},\quad\quad u\neq
            a
            ,\\
            \frac{\mathrm{d}w}{\mathrm{d} u}&=\left\{\frac{(1-a)(1-2a)}{b^3\sqrt{b^2-4a+4a^2}}\right\}^{1/2},\qquad u=a,
        \end{split}
\end{equation}
where we again have used L'H$\hat{o}$spital's rule for $u=a$.\\

By (\ref{IP after the t mapping 2}) and (\ref{mapping w to u when
a<0}), we have
\begin{equation}\label{}
\begin{split}
t_n(x,N+1)=& \frac{(-1)^n}{2\pi
i}\frac{\Gamma(n+N+2)e^nb^{-n}e^{N\gamma}}{\Gamma(n+1)\Gamma(N-n+1)}\\
&\times
\int_{0}^{+\infty}\int_{\infty}^{\left(0^+\right)}\frac{h(u,\tau)}{u}e^{N
\left(a\ln(-u)-u\right)}e^{N(b\ln
\tau-\tau)}\mathrm{d}u\mathrm{d}\tau,
\end{split}
\end{equation}
where
\begin{equation}\label{}
    h(u,\tau)=\frac{u}{w-1}\frac{\mathrm{d}t}{\mathrm{d}
\tau}\frac{\mathrm{d}w}{\mathrm{d} u}.
\end{equation}
Define recursively
\begin{eqnarray}
  h_l(u,\tau) &=& a_l(\tau)+(u-a)g_l(u,\tau),\\
  h_{l+1}(u,\tau)&=&u\frac{\partial
g_l(u,\tau)}{\partial u},
\end{eqnarray}
where $h_0(u,\tau)=h(u,\tau)$, and expand $a_l(\tau)$ into a
Maclaurin series
\begin{equation}\label{al in a<0}
       a_l(\tau)=\sum_{j=0}^{\infty} a_{l,j}\tau^j.
\end{equation}
 By an integration-by-parts procedure, we obtain
\begin{equation}\label{expansion for a<0 case 1}
  t_n(x,N+1) =\frac{(-1)^{n}\Gamma(n+N+2)e^nb^{-n}e^{N\gamma}}{\Gamma(N-n+1)\Gamma(-aN+1)N^{aN+n+1}}
  \Biggr[\sum_{l=0}^{p-1}\frac{c_l}{N^l}+\varepsilon_p^-\Biggr],
\end{equation}
where
\begin{equation}
  c_l =\frac{N^{n+1}}{\Gamma(n+1)}\int_0^\infty
  a_l(\tau)\tau^ne^{-N\tau}\mathrm{d}\tau\sim
  \sum_{m=0}^\infty a_{l,m}\frac{\Gamma(n+m+1)}{\Gamma(n+1)N^m}
  \end{equation}
 and
\begin{equation}\label{error in a<0}
\varepsilon_p^-= \frac{N^{aN+n+1-p}\Gamma(-aN+1)}{2\Gamma(n+1)\pi
i}\int_{0}^{+\infty}\int_{-\infty}^{\left(0^+\right)}\frac{h_p(u,\tau)}{u}e^{N
\widetilde{\psi}(u)}\tau ^n e^{-N\tau}\mathrm{d}u\mathrm{d}\tau.
\end{equation}
Estimation of the error term given in (\ref{error in a<0}) is
exactly the same as what we have done in \cite{pan}.

\bibliographystyle{plain}

\end{document}